\documentclass{eusflat2019}
\usepackage{amsthm}
\newtheorem{theorem}{Theorem}[section]
\newtheorem{corollary}[theorem]{Corollary}
\newtheorem{lemma}[theorem]{Lemma}
\newtheorem{proposition}[theorem]{Proposition}
\newtheorem{definition}[theorem]{Definition}

\newcommand{\Prop}{\mathsf{Prop}}
\newcommand{\X}{\mathbb{X}}

\newcommand{\val}[1]{[\![{#1}]\!]}
\newcommand{\descr}[1]{(\![{#1}]\!)}

\renewcommand{\phi}{\varphi}
\newcommand{\cnomm}{\mathbf{m}}
\newcommand{\nomj}{\mathbf{j}}

\usepackage{lscape}
\usepackage{pdflscape}

\usepackage{stmaryrd}

\usepackage{cmll}
\usepackage[table]{xcolor}
\usepackage{comment}

\usepackage{mathdots}

\usepackage{amscd}
\usepackage{amssymb}
\usepackage{amsmath}
\usepackage{mathptmx}
\usepackage{mathrsfs}
\usepackage{color}
\usepackage{xspace}
\usepackage{bussproofs}
\EnableBpAbbreviations

\usepackage{tikz}

\usepackage{txfonts}

\usepackage{centernot}

\usepackage{mathtools}

\usepackage{relsize}

\usepackage{multirow}

\usepackage{multicol}

\usepackage{array}
\newcolumntype{C}[1]{>{\centering\arraybackslash}p{#1}}
\newcolumntype{L}[1]{>{\arraybackslash}p{#1}}

\usepackage{hyperref}

\newtheorem{remark}[theorem]{Remark}

\title{\bf Modelling competing theories}

\author{{\bf Willem Conradie$^a$} and {\bf Andrew Craig$^b$} and {\bf Alessandra Palmigiano$^{b,c}$} and {\bf Nachoem Wijnberg$^{d,e}$}\\
$^a$School of Mathematics, University of the Witwatersrand \\ $^b$Department of Mathematics and Applied Mathematics, University of Johannesburg, \\
$^c$Faculty of Technology, Policy and Management, Delft University of Technology\\
$^d$ College of Business and Economics,  University of Johannesburg\\
$^e$ Amsterdam Business School, University of Amsterdam}

\begin{document}
\maketitle
\begin{abstract}
We introduce a complete many-valued semantics for two normal lattice-based modal logics. This semantics is based on reflexive many-valued graphs. We discuss an interpretation and possible applications of this logical framework in the context of the formal analysis of the interaction between (competing) scientific theories.
\end{abstract}

\section{Introduction}
The contributions of this paper lie at the intersection of several strands of research. They are rooted in  the generalized Sahlqvist theory  for normal LE-logics \cite{CoPa:non-dist, conradie2016constructive}, i.e.~those logics algebraically captured by varieties of normal lattice expansions (LEs) \cite{gehrke2001bounded}. Via canonical extensions and discrete duality, basic normal LE-logics of arbitrary signatures and  a large class of their axiomatic extensions can be uniformly endowed with complete relational semantics of different kinds, of which  those of interest to the present paper are relational structures based on {\em formal contexts} \cite{gehrke2006generalized, galatos2013residuated, conradie2016categories,Tarkpaper, greco2018algebraic} and {\em reflexive graphs} \cite{conradie2015relational, graph-based-wollic}. In a mathematical setting in which the original discrete duality for perfect normal LEs has been relaxed to a discrete adjunction for complete normal LEs, these semantic structures have yielded uniform theoretical developments in the algebraic proof theory \cite{greco2018algebraic} and in the model theory \cite{conradie2018goldblatt} of LE-logics, and also insights on possible interpretations of LE-logics which have generated new opportunities for applications. In particular, via polarity-based semantics, in \cite{conradie2016categories}, the basic non-distributive modal logic and some of its axiomatic extensions are interpreted as  {\em epistemic logics of categories and concepts}, and in \cite{Tarkpaper}, the corresponding `common knowledge'-type construction is used to give an epistemic-logical formalization of the notion of {\em prototype} of a category; in \cite{roughconcepts, ICLA2019paper}, polarity-based semantics for non-distributive modal logic is proposed as an encompassing framework for the integration of rough set theory \cite{pawlak} and formal concept analysis \cite{ganter2012formal}, and in this context, the basic non-distributive modal logic is interpreted as  the logic of {\em rough concepts};    via  graph-based semantics, in \cite{graph-based-wollic}, the same logic is interpreted as the logic of {\em informational entropy}, i.e.~an inherent boundary to knowability due e.g.~to perceptual, theoretical, evidential or linguistic limits. In the graphs $(Z, E)$ on which the relational structures are based, the relation $E$  is interpreted as the indiscernibility relation induced by informational entropy, much in the same style as Pawlak's  approximation spaces in rough set theory. However, the key difference is that,  rather than generating modal operators which associate any subset of $Z$ with its definable $E$-approximations,  $E$ generates a complete lattice (i.e.~the lattice of $E^c$-concepts). In this approach, concepts are not definable approximations of predicates, but rather they represent `all there is to know',  i.e.~the theoretical horizon to knowability, given the inherent boundary  encoded into $E$ (in their turn, $E^c$-concepts are approximated by means of the additional relations of the graph-based relational structures from which the semantic modal operators arise). Interestingly, $E$ is required to be reflexive but in general neither transitive nor symmetric, which is in line with proposals in rough set theory \cite{wybraniec1989generalization,yao1996-Sahlqvist,vakarelov2005modal} that indiscernibility does not need to give rise to equivalence relations.

In this paper, we start exploring the many-valued version of the graph-based semantics of  \cite{graph-based-wollic} for two axiomatic extensions of the basic normal non-distributive modal logic, and in particular their potential for modelling situations in which informational entropy derives from the theoretical frameworks under which empirical studies are conducted. 
\section{Preliminaries}
This section is based on \cite[Section 2.1]{graph-based-wollic} and \cite[Section 7.2]{roughconcepts}.
\subsection{Basic normal nondistributive modal logic}
\label{sec:logics}
Let $\Prop$ be a (countable or finite) set of atomic propositions. The language $\mathcal{L}$ of the {\em basic normal nondistributive modal logic} is defined as follows:
\[ \varphi := \bot \mid \top \mid p \mid  \varphi \wedge \varphi \mid \varphi \vee \varphi \mid \Box \varphi \mid  \Diamond\varphi,\] 
where $p\in \Prop$. 
The {\em basic}, or {\em minimal normal} $\mathcal{L}$-{\em logic} is a set $\mathbf{L}$ of sequents $\phi\vdash\psi$  with $\phi,\psi\in\mathcal{L}$, containing the following axioms:
		{\small{
			\begin{align*}
				&p\vdash p, && \bot\vdash p, && p\vdash \top, & &  &\\
				&p\vdash p\vee q, && q\vdash p\vee q, && p\wedge q\vdash p,\\ 
				&p\wedge q\vdash q, && \top\vdash \Box \top, && \Box p\wedge \Box q \vdash \Box ( p\wedge q),\\
                &\Diamond \bot\vdash \bot, && \Diamond p \vee \Diamond  q \vdash \Diamond ( p \vee q) &
			\end{align*}
			}}
	%
	%
		and closed under the following inference rules:
		{\small{
		\begin{displaymath}
			\frac{\phi\vdash \chi\quad \chi\vdash \psi}{\phi\vdash \psi}
			\quad
			\frac{\phi\vdash \psi}{\phi\left(\chi/p\right)\vdash\psi\left(\chi/p\right)}
		\end{displaymath}
		\begin{displaymath}
			\frac{\chi\vdash\phi\quad \chi\vdash\psi}{\chi\vdash \phi\wedge\psi}
			\quad
			\frac{\phi\vdash\chi\quad \psi\vdash\chi}{\phi\vee\psi\vdash\chi}
		\end{displaymath}
		\begin{displaymath}
			\frac{\phi\vdash\psi}{\Box \phi\vdash \Box \psi}
\quad
\frac{\phi\vdash\psi}{\Diamond \phi\vdash \Diamond \psi}
		\end{displaymath}
		}}
An {\em $\mathcal{L}$-logic} is any  extension of $\mathbf{L}$  with $\mathcal{L}$-axioms $\phi\vdash\psi$. Relevant to what follows are the axiomatic extensions of $\mathbf{L}$ generated by $\Box\bot\vdash \bot$ and $\top\vdash \Diamond\top$, and by $\Box\phi\vdash \phi$ and $\phi\vdash\Diamond\phi$. Let $\mathbf{L}_0$ (resp.\ $\mathbf{L}_1$) be the axiomatic extension obtained by adding $\Box\bot\vdash \bot$ (resp.\ $\Box p\vdash p$) to $\mathbf{L}$. Notice that $\mathbf{L}_1$ is an extension of $\mathbf{L}_0$.

\subsection{Many-valued enriched formal contexts}
Throughout this paper,  we let $\mathbf{A} = (D, 1, 0, \vee, \wedge, \otimes, \to)$ denote an arbitrary but fixed 
complete frame-distributive and dually frame-distributive, 
commutative and associative residuated lattice  (understood as the algebra of truth-values) such that $1\to \alpha = \alpha$ for every $\alpha\in D$. 
For every set $W$, an $\mathbf{A}$-{\em valued subset}  (or $\mathbf{A}$-{\em subset}) of $W$ is a map $u: W\to \mathbf{A}$.  We let $\mathbf{A}^W$ denote the set of all $\mathbf{A}$-subsets. Clearly, $\mathbf{A}^W$ inherits the algebraic structure of $\mathbf{A}$ by defining the operations and the order pointwise. The $\mathbf{A}$-{\em subsethood} relation between elements of $\mathbf{A}^W$ is the map $S_W:\mathbf{A}^W\times \mathbf{A}^W\to \mathbf{A}$ defined as $S_W(f, g) :=\bigwedge_{z\in W }(f(z)\rightarrow g(z)) $. For every $\alpha\in \mathbf{A}$, 
let $\{\alpha/ w\}: W\to \mathbf{A}$ be defined by $v\mapsto \alpha$ if $v = w$ and $v\mapsto \bot^{\mathbf{A}}$ if $v\neq w$. Then, for every $f\in \mathbf{A}^W$,
\begin{equation}\label{eq:MV:join:generators}
f = \bigvee_{w\in W}\{f(w)/ w\}.
\end{equation}
 When $u, v: W\to \mathbf{A}$ and $u\leq v$ w.r.t.~the pointwise order, we write $u\subseteq v$.
An $\mathbf{A}$-{\em valued relation} (or $\mathbf{A}$-{\em relation}) is a map $R: U \times W \rightarrow \mathbf{A}$. Two-valued relations can be regarded as  $\mathbf{A}$-relations. In particular for any set $Z$, we let $\Delta_Z: Z\times Z\to \mathbf{A}$ be defined by $\Delta_Z(z, z') = \top$ if $z = z'$ and $\Delta_Z(z, z') = \bot$ if $z\neq z'$. An $\mathbf{A}$-relation $R: Z\times Z\to \mathbf{A}$ is {\em reflexive} if $\Delta_Z \subseteq R$.
Any $\mathbf{A}$-valued relation $R: U \times W \rightarrow \mathbf{A}$ induces  maps $R^{(0)}[-] : \mathbf{A}^W \rightarrow \mathbf{A}^U$ and $R^{(1)}[-] : \mathbf{A}^U \rightarrow \mathbf{A}^W$ defined as follows: for every $f: U \to \mathbf{A}$ and every $u: W \to \mathbf{A}$,
\begin{center}
	\begin{tabular}{r l }
$R^{(1)}[f]:$ & $ W\to \mathbf{A}$\\
		
		& $ x\mapsto \bigwedge_{a\in U}(f(a)\rightarrow R(a, x))$\\
	&\\	
 $R^{(0)}[u]: $ & $U\to \mathbf{A} $\\
		& $a\mapsto \bigwedge_{x\in W}(u(x)\rightarrow R(a, x))$\\
	\end{tabular}
\end{center}

A {\em formal}  $\mathbf{A}$-{\em context}\footnote{ In the crisp setting, a {\em formal context} \cite{ganter2012formal}, or {\em polarity},  is a structure $\mathbb{P} = (A, X, I)$ such that $A$ and $X$ are sets, and $I\subseteq A\times X$ is a binary relation. Every such $\mathbb{P}$ induces maps $(\cdot)^\uparrow: \mathcal{P}(A)\to \mathcal{P}(X)$ and $(\cdot)^\downarrow: \mathcal{P}(X)\to \mathcal{P}(A)$, respectively defined by the assignments $B^\uparrow: = I^{(1)}[B]$ and $Y^\downarrow: = I^{(0)}[Y]$. A {\em formal concept} of $\mathbb{P}$ is a pair 
$c = (\val{c}, \descr{c})$ such that $\val{c}\subseteq A$, $\descr{c}\subseteq X$, and $\val{c}^{\uparrow} = \descr{c}$ and $\descr{c}^{\downarrow} = \val{c}$.   The set $L(\mathbb{P})$  of the formal concepts of $\mathbb{P}$ can be partially ordered as follows: for any $c, d\in L(\mathbb{P})$, \[c\leq d\quad \mbox{ iff }\quad \val{c}\subseteq \val{d} \quad \mbox{ iff }\quad \descr{d}\subseteq \descr{c}.\]
With this order, $L(\mathbb{P})$ is a complete lattice, the {\em concept lattice} $\mathbb{P}^+$ of $\mathbb{P}$. Any complete lattice $\mathbb{L}$ is isomorphic to the concept lattice $\mathbb{P}^+$ of some polarity $\mathbb{P}$.} or $\mathbf{A}$-{\em polarity} (cf.~\cite{belohlavek}) is a structure $\mathbb{P} = (A, X, I)$ such that $A$ and $X$ are sets and $I: A\times X\to \mathbf{A}$. Any formal $\mathbf{A}$-context induces  maps $(\cdot)^{\uparrow}: \mathbf{A}^A\to \mathbf{A}^X$ and $(\cdot)^{\downarrow}: \mathbf{A}^X\to \mathbf{A}^A$ given by $(\cdot)^{\uparrow} = I^{(1)}[\cdot]$ and $(\cdot)^{\downarrow} = I^{(0)}[\cdot]$. 
%
These maps are such that, for every $f\in \mathbf{A}^A$ and every $u\in \mathbf{A}^X$,
\[S_A(f, u^{\downarrow}) = S_X(u, f^{\uparrow}),\]
that is, the pair of maps $(\cdot)^{\uparrow}$ and $(\cdot)^{\downarrow}$ form an $\mathbf{A}$-{\em Galois connection}.
In \cite[Lemma 5]{belohlavek}, it is shown that every  $\mathbf{A}$-Galois connection arises from some formal  $\mathbf{A}$-context. A {\em formal}  $\mathbf{A}$-{\em concept} of $\mathbb{P}$ is a pair $(f, u)\in \mathbf{A}^A\times \mathbf{A}^X$ such that $f^{\uparrow} = u$ and $u^{\downarrow} = f$. It follows immediately from this definition that if $(f, u)$ is a formal $\mathbf{A}$-concept, then $f^{\uparrow \downarrow} = f$ and $u^{\downarrow\uparrow} = u$, that is, $f$ and $u$ are {\em stable}. The set of formal $\mathbf{A}$-concepts can be partially ordered as follows:
\[(f, u)\leq (g, v)\quad \mbox{ iff }\quad f\subseteq g \quad \mbox{ iff }\quad v\subseteq u. \]
Ordered in this way, the set of the formal  $\mathbf{A}$-concepts of $\mathbb{P}$ is a complete lattice, which we denote $\mathbb{P}^+$. 

	An {\em enriched formal $\mathbf{A}$-context} (cf.~\cite[Section 7.2]{roughconcepts}) is a structure  $\mathbb{F} = (\mathbb{P}, R_\Box, R_\Diamond)$ such that $\mathbb{P} = (A, X, I)$ is a formal  $\mathbf{A}$-context and $R_\Box: A\times X\to \mathbf{A}$ and $R_\Diamond: X\times A\to \mathbf{A}$ are $I$-{\em compatible}, i.e.~$R_{\Box}^{(0)}[\{\alpha / x\}]$, $R_{\Box}^{(1)}[\{\alpha / a\}]$,  $R_{\Diamond}^{(0)}[\{\alpha / a\}]$ and $R_{\Diamond}^{(1)}[\{\alpha / x\}]$ are stable for every $\alpha \in \mathbf{A}$, $a \in A$ and $x \in X$.   
	The {\em complex algebra} of an  enriched formal $\mathbf{A}$-context $\mathbb{F} = (\mathbb{P}, R_\Box, R_\Diamond)$ is the algebra $\mathbb{F}^{+} = (\mathbb{P}^{+}, [R_{\Box}], \langle R_{\Diamond} \rangle )$ where $[R_{\Box}], \langle R_{\Diamond} \rangle : \mathbb{P}^{+} \to \mathbb{P}^{+}$ are defined by the following assignments: for every $c = (\val{c}, \descr{c}) \in \mathbb{P}^{+}$, 
	\begin{center}
	\begin{tabular}{l cl}
		$[R_{\Box}]c$ & $ =$&$  (R_{\Box}^{(0)}[\descr{c}], (R_{\Box}^{(0)}[\descr{c}])^{\uparrow})$\\
		$ \langle R_{\Diamond} \rangle c $ & $ =$&$  ((R_{\Diamond}^{(0)}[\val{c}])^{\downarrow}, R_{\Diamond}^{(0)}[\val{c}])$.\\
	\end{tabular}
	\end{center}

\begin{lemma}\label{prop:fplus}
(cf.~\cite[Lemma 15]{roughconcepts}) If $\mathbb{F} = (\mathbb{X}, R_{\Box}, R_{\Diamond})$	is an enriched formal   $\mathbf{A}$-context, $\mathbb{F}^+ = (\mathbb{X}^+, [R_{\Box}], \langle R_{\Diamond}\rangle)$ is a complete  normal lattice expansion such that $[R_\Box]$ is completely meet-preserving and $\langle R_\Diamond\rangle$ is completely join-preserving.
\end{lemma}

\section{Many-valued graph-based frames}\label{sec:frames}
A reflexive $\mathbf{A}$-{\em graph} is a structure $\X = (Z, E)$ such that $Z$ is a nonempty set, and $E: Z\times Z\to \mathbf{A}$ is reflexive. From now on, we will assume that all $\mathbf{A}$-graphs we consider are reflexive even when we drop the adjective.  
\begin{definition}
\label{def:A-lifting of a graph}
For any reflexive $\mathbf{A}$-graph $\X = (Z, E)$,  the formal $\mathbf{A}$-context associated with $\X$ is
\[\mathbb{P_X}: = (Z_A, Z_X, I_E),\]
where $Z_A := \mathbf{A}\times Z$ and $Z_X :=  Z$, and  $I_E: Z_A\times Z_X\to \mathbf{A}$ is defined by $I_E((\alpha, z), z') = E(z, z')\rightarrow \alpha$. We let $\X^+: = \mathbb{P_X}^+$.\end{definition}
Any  $R: Z\times Z\to \mathbf{A}$  admits the following {\em liftings}:
\begin{center}
	\begin{tabular}{r l }
$I_R:$ & $ Z_A\times Z_X \to \mathbf{A}$\\
		
		& $ ((\alpha, z), z')\mapsto R(z, z')\rightarrow \alpha$\\ 
	&\\	
$J_R:$ & $ Z_X\times Z_A \to \mathbf{A}$\\
		
		& $ (z, (\alpha, z'))\mapsto R(z, z')\rightarrow \alpha$\\ 
	\end{tabular}
\end{center}
Recall that for all $f: \mathbf{A}\times Z\to \mathbf{A}$, and $u: Z\to \mathbf{A}$, the maps\footnote{\label{footnote: abbreviations} Applying this notation to a graph $\X = (Z, E)$, we will abbreviate $E^{[0]}[u]$ and $E^{[1]}[f]$ as $u^{[0]}$ and $f^{[1]}$, respectively, for each $u, f$ as above,  
and write $u^{[01]}$ and $f^{[10]}$ for $(u^{[0]})^{[1]}$ and $(f^{[1]})^{[0]}$, respectively. Then $u^{[0]} = I_{E}^{(0)}[u] = u^{\downarrow}$ and $f^{[1]} = I_{E}^{(1)}[f] = f^{\uparrow}$, where the maps $(\cdot)^\downarrow$ and $(\cdot)^\uparrow$ are those associated with the polarity $\mathbb{P_X}$.} 
$f^{\uparrow}:  Z\to \mathbf{A}$ and $u^{\downarrow}: \mathbf{A}\times Z\to \mathbf{A}$ are respectively defined by the assignments \[ z\mapsto \bigwedge_{(\alpha, z')\in Z_A}[ f(\alpha, z')\to (E(z', z)\to\alpha)]\]  \[(\alpha, z)\mapsto \bigwedge_{z'\in Z_X}[ u(z')\to (E(z, z')\to\alpha)].\] 
\begin{definition}\label{def:graph:based:frame:and:model}
A {\em graph-based $\mathbf{A}$-frame}  is a structure $\mathbb{G} = (\mathbb{X},  R_{\Diamond}, R_{\Box})$ where $\mathbb{X} = (Z,E)$ is a reflexive $\mathbf{A}$-graph, and  $R_{\Diamond}$ and $R_{\Box}$  are binary $\mathbf{A}$-relations on $Z$ such that the structure $\mathbb{F_G}: = (\mathbb{P_X}, I_{R_\Box}, J_{R_\Diamond})$ is an enriched formal $\mathbf{A}$-context. That is, 
$R_{\Diamond}$ and $R_{\Box}$ satisfy the following  $E$-{\em compatibility} conditions: for any $z\in Z$ and $\alpha, \beta\in\mathbf{A}$,
\begin{align*}
(R_\Box^{[0]}[\{\beta /   z \}])^{[10]} &\subseteq R_\Box^{[0]}[\{\beta /  z \}] \\
(R_\Box^{[1]}[\{\beta /  (\alpha, z) \}])^{[01]} &\subseteq R_\Box^{[1]}[\{\beta / (\alpha, z) \}]\\
(R_\Diamond^{[1]}[\{\beta /  z\}])^{[10]} &\subseteq R_\Diamond^{[1]}[\{\beta /   z \}] \\
(R_\Diamond^{[0]}[\{\beta /  (\alpha, z) \}])^{[01]} &\subseteq R_\Diamond^{[0]}[\{\beta / (\alpha, z) \}].
\end{align*}
where for all $f: \mathbf{A}\times Z\to \mathbf{A}$ and $u: Z\to \mathbf{A}$,
\begin{center}
	\begin{tabular}{ll}
$u^{[0]} = E^{[0]}[u]:$ & $\mathbf{A}\times Z\to \mathbf{A}$\\
& $(\alpha, z)\mapsto I_E^{(0)}[u](\alpha, z) = u^{\downarrow}(\alpha, z)$\\
&\\
$f^{[1]} = E^{[1]}[f]:$ & $Z\to \mathbf{A}$\\
& $z\mapsto I_E^{(1)}[f](z) = f^{\uparrow}(z)$\\
\end{tabular}
	\end{center}
\begin{center}
	\begin{tabular}{ll}
$R_{\Box}^{[0]}[u]:$ & $\mathbf{A}\times Z\to \mathbf{A}$\\
& $(\alpha, z)\mapsto I_{R_{\Box}}^{(0)}[u](\alpha, z)$\\
&\\
$R_{\Box}^{[1]}[f]:$ & $ Z\to \mathbf{A}$\\
& $z\mapsto I_{R_\Box}^{(1)}[f](z)$\\
\end{tabular}
	\end{center}
\begin{center}
	\begin{tabular}{ll}
$R_{\Diamond}^{[0]}[f]:$ & $Z\to \mathbf{A}$\\
& $z\mapsto  J_{R_\Diamond}^{(0)}[f]( z)$\\
&\\	
$R_{\Diamond}^{[1]}[u]:$ & $\mathbf{A}\times Z\to \mathbf{A}$\\
& $(\alpha, z)\mapsto  J_{R_\Diamond}^{(1)}[u](\alpha, z)$.\\
\end{tabular}
	\end{center}
	Hence, for any $z\in Z$ and $\alpha\in \mathbf{A}$,
	\begin{center}
	\begin{tabular}{l}
	$E^{[0]}[u](\alpha, z): = \bigwedge_{z'\in Z_X}[u(z')\to (E(z, z') \to \alpha)]$\\
	$E^{[1]}[f](z): = \bigwedge_{(\alpha, z')\in Z_A}[f(\alpha, z')\to  (E(z', z) \to \alpha)]$.\\

	$R_{\Box}^{[0]}[u](\alpha, z): = \bigwedge_{z'\in Z_X}[u( z')\to (R_{\Box}(z, z') \to \alpha)]$\\
	$R_{\Box}^{[1]}[f](z): = \bigwedge_{(\alpha, z')\in Z_A}[f(\alpha, z')\to  (R_{\Box}(z', z) \to \alpha)]$\\
		$R_{\Diamond}^{[0]}[f](z): = \bigwedge_{(\alpha, z')\in Z_A}[f(\alpha, z')\to  (R_{\Diamond}(z, z') \to \alpha)]$\\
	$R_{\Diamond}^{[1]}[u](\alpha, z): = \bigwedge_{z'\in Z_X}[u(z')\to (R_{\Diamond}(z', z) \to \alpha)]$.\\

	\end{tabular}
	\end{center}
	
The {\em complex algebra} of a graph-based $\mathbf{A}$-frame $\mathbb{G}= (\mathbb{X},  R_{\Diamond}, R_{\Box})$ is the algebra $\mathbb{G}^+ = (\mathbb{X}^+, [R_\Box], \langle R_\Diamond\rangle),$
	where $\mathbb{X}^+: = \mathbb{P_X}^+$, and $[R_\Box]$ and $\langle R_\Diamond\rangle$ are unary operations on $\mathbb{X}^+$ defined as follows: for every $c = (\val{c}, \descr{c}) \in \mathbb{X}^+$,
	\begin{center}
	\begin{tabular}{l cl}
	$[R_\Box]c$ & $ =$&$   (R_{\Box}^{[0]}[\descr{c}], (R_{\Box}^{[0]}[\descr{c}])^{[1]})$\\
	 $ \langle R_\Diamond\rangle c$ & $ =$&$  ((R_{\Diamond}^{[0]}[\val{c}])^{[0]}, R_{\Diamond}^{[0]}[\val{c}])$.\\
\end{tabular}
	\end{center}
 \end{definition}
 By definition, it immediately follows that
 \begin{lemma}
If  $\mathbb{G}$ is a  graph-based $\mathbf{A}$-frame,  $\mathbb{G}^+ = \mathbb{F_G}^+$.
 \end{lemma}
 Hence, by the lemma above and Lemma \ref{prop:fplus}, 
\begin{lemma}
If $\mathbb{G} = (\mathbb{X}, R_{\Box}, R_{\Diamond})$	is a  graph-based $\mathbf{A}$-frame, $\mathbb{G}^+ = (\mathbb{X}^+, [R_{\Box}], \langle R_{\Diamond}\rangle)$ is a complete  normal lattice expansion such that $[R_\Box]$ is completely meet-preserving and $\langle R_\Diamond\rangle$ is completely join-preserving.
\end{lemma}

The following lemma is an immediate consequence of \cite[Lemma 14]{roughconcepts} applied to $\mathbb{F_G}$.

\begin{lemma}\label{equivalence of I-compatible-mv}
For every graph-based $\mathbf{A}$-graph $\mathbb{G} = (\mathbb{X}, R_{\Box}, R_{\Diamond})$,
\begin{enumerate}
\item the following are equivalent:	
\begin{enumerate}
		\item [(i)] $(R_{\Box}^{[0]}[\{\alpha/z\}])^{[10]}\subseteq R_{\Box}^{[0]}[\{\alpha/z\}]$  for every $z\in Z$ and $\alpha\in\mathbf{A}$;
		
		\item [(ii)]  $(R_{\Box}^{[0]}[u])^{[10]}\subseteq R_{\Box}^{[0]}[u]$  for every $u: Z_X\to\mathbf{A}$;
		\item [(iii)] $R_{\Box}^{[1]}[f^{[10]}]\subseteq R_{\Box}^{[1]}[f]$
		for every  $f:Z_A\to\mathbf{A}$.
	\end{enumerate}
\item the following are equivalent:	
\begin{enumerate}
		\item [(i)] $(R_{\Box}^{[1]}[\{\alpha/(\beta, z)\}])^{[01]}\subseteq R_{\Box}^{[1]}[\{\alpha/(\beta, z)\}]$  for every $z\in Z$ and $\alpha, \beta\in\mathbf{A}$;
		
		\item [(ii)]  $(R_{\Box}^{[1]}[f])^{[01]}\subseteq R_{\Box}^{[1]}[f]$  for every $f: Z_A\to\mathbf{A}$;
		\item [(iii)] $R_{\Box}^{[0]}[u^{[01]}]\subseteq R_{\Box}^{[0]}[u]$
		for every  $u:Z_X\to\mathbf{A}$.
	\end{enumerate}
\item the following are equivalent:	
\begin{enumerate}
		\item [(i)] $(R_{\Diamond}^{[0]}[\{\alpha/(\beta, z)\}])^{[01]}\subseteq R_{\Diamond}^{[0]}[\{\alpha/(\beta, z)\}]$  for every $z\in Z$ and $\alpha, \beta\in\mathbf{A}$;
		
		\item [(ii)]  $(R_{\Diamond}^{[0]}[f])^{[01]}\subseteq R_{\Diamond}^{[0]}[f]$  for every $f: Z_A\to\mathbf{A}$;
		\item [(iii)] $R_{\Diamond}^{[1]}[u^{[01]}]\subseteq R_{\Diamond}^{[1]}[u]$
		for every  $u:Z_X\to\mathbf{A}$.
	\end{enumerate}	
	
\item the following are equivalent :	
\begin{enumerate}
		\item [(i)] $(R_{\Diamond}^{[1]}[\{\alpha/z\}])^{[10]}\subseteq R_{\Diamond}^{[1]}[\{\alpha/z\}]$  for every $z\in Z$ and $\alpha\in\mathbf{A}$;
		
		\item [(ii)]  $(R_{\Diamond}^{[1]}[u])^{[10]}\subseteq R_{\Diamond}^{[1]}[u]$  for every $u: Z_X\to\mathbf{A}$;
		\item [(iii)] $R_{\Diamond}^{[0]}[f^{[10]}]\subseteq R_{\Diamond}^{[0]}[f]$
		for every  $f:Z_A\to\mathbf{A}$.
	\end{enumerate}	
\end{enumerate}
\end{lemma}

\section{Many-valued graph-based models}
Let $\mathcal{L}$ be the language of Section \ref{sec:logics}. 
\begin{definition}
A {\em graph-based}  $\mathbf{A}$-{\em model} of $\mathcal{L}$  is a tuple $\mathbb{M} = (\mathbb{G}, V)$ such that $\mathbb{G} = (\mathbb{X}, R_\Box, R_\Diamond)$ is a graph-based $\mathbf{A}$-frame and $V: \mathcal{L}\to \mathbb{G}^+$ is a homomorphism. For every $\phi\in \mathcal{L}$, let $V(\phi): = (\val{\phi}, \descr{\phi})$, where $\val{\phi}: Z_A\to \mathbf{A}$ and $\descr{\phi}: Z_X\to \mathbf{A}$ are s.t.~$\val{\phi}^{[1]} = \descr{\phi}$ and $\descr{\phi}^{[0]} = \val{\phi}$. Hence:
%
\begin{center}
\begin{tabular}{r c l}
$V(p)$ & = & $(\val{p}, \descr{p})$\\
$V(\top)$ & = & $(1^{\mathbf{A}^{Z_A}}, (1^{\mathbf{A}^{Z_A}})^{[1]})$\\
$V(\bot)$ & = & $((1^{\mathbf{A}^{Z_X}})^{[0]}, 1^{\mathbf{A}^{Z_X}})$\\
$V(\phi\wedge \psi)$ & = & $(\val{\phi}\wedge\val{\psi}, (\val{\phi}\wedge\val{\psi})^{[1]})$\\
$V(\phi\vee \psi)$ & = & $((\descr{\phi}\wedge\descr{\psi})^{[0]}, \descr{\phi}\wedge\descr{\psi})$\\
$V(\Box\phi)$ & = & $(R^{[0]}_\Box[\descr{\phi}], (R^{[0]}_\Box[\descr{\phi}])^{[1]})$\\
$V(\Diamond\phi)$ & = & $((R^{[0]}_\Diamond[\val{\phi}])^{[0]}, R^{[0]}_\Diamond[\val{\phi}])$.\\
\end{tabular}
\end{center}
Valuations induce  $\alpha$-{\em support relations} between value-state pairs and formulas for each $\alpha\in \mathbf{A}$ (in symbols: $\mathbb{M}, (\beta, z)\Vdash^\alpha \phi$), and $\alpha$-{\em refutation relations} between states of models and formulas for each $\alpha\in \mathbf{A}$ (in symbols: $\mathbb{M},  z\succ^\alpha \phi$) such that for every $\phi\in \mathcal{L}$, all $z\in Z$and all $\beta\in \mathbf{A}$,
\begin{center}
	\begin{tabular}{r c l}
$\mathbb{M}, (\beta, z)\Vdash^\alpha \phi$ & iff  &$ \alpha\leq \val{\phi}(\beta, z),$\\
$\mathbb{M},z \succ^\alpha \phi $ & iff  &$ \alpha\leq \descr{\phi}( z).$\\
\end{tabular}
\end{center}
This can be equivalently expressed as follows:
{\small{
\begin{center}
\begin{tabular}{l c l}
$\mathbb{M}, (\beta, z)\Vdash^\alpha p$ &iff& $\alpha\leq \val{p}(\beta, z)$;\\
$\mathbb{M}, (\beta, z)\Vdash^\alpha \top$ &iff& $\alpha\leq 1^{\mathbf{A}^{Z_A}}(\beta, z)$ i.e.~always;\\
$\mathbb{M}, (\beta, z)\Vdash^\alpha \bot$ &iff& $\alpha\leq (1^{\mathbf{A}^{Z_X}})^{[0]} (\beta, z)$\\
&& $ =\bigwedge_{z'\in Z_X}[1^{\mathbf{A}^{Z_X}}(z')\to (E(z, z') \to \beta)]$\\
&& $ =\bigwedge_{z'\in Z_X}[E(z, z') \to \beta]$\\
&& $ =\beta$;\\
$\mathbb{M}, (\beta, z)\Vdash^\alpha \phi\wedge \psi$ &iff& $\mathbb{M}, (\beta, z)\Vdash^\alpha \phi\, $ and $\, \mathbb{M}, (\beta, z)\Vdash^\alpha \psi$;\\
$\mathbb{M}, (\beta, z)\Vdash^\alpha \phi\vee \psi$ &iff& $\alpha\leq (\descr{\phi}\wedge\descr{\psi})^{[0]}(\beta, z)$\\
&& $ = \bigwedge_{z'\in Z_X}[(\descr{\phi}\wedge\descr{\psi})(z')$\\ 
&&$\quad \to (E(z, z') \to \beta)]$;\\
$\mathbb{M}, (\beta, z)\Vdash^\alpha \Box \phi$ &iff& $\alpha\leq (R^{[0]}_\Box[\descr{\phi}])(\beta, z)$\\
&& $ = \bigwedge_{ z'\in Z_X}[\descr{\phi}(z')\to (R_{\Box}(z, z')$\\
&&$\quad \to \beta)]$;\\
$\mathbb{M}, (\beta, z)\Vdash^\alpha \Diamond \phi$ &iff& $\alpha\leq ((R^{[0]}_\Diamond[\val{\phi}])^{[0]})(\beta, z)$\\
&&$ = \bigwedge_{z'\in Z_X}[R^{[0]}_\Diamond[\val{\phi}](z')\to (E(z, z')$\\
&&$\quad \to \beta)]$;\\
\end{tabular}
\end{center}

\begin{center}
\begin{tabular}{l c l}
$\mathbb{M}, z\succ^\alpha p$ & iff & $\alpha\leq \descr{p}( z)$;\\
$\mathbb{M}, z\succ^\alpha \bot$ & iff & $\alpha\leq 1^{\mathbf{A}^{Z_X}}( z)$ i.e.~always;\\
$\mathbb{M},  z\succ^\alpha \top$ & iff & $\alpha\leq (1^{\mathbf{A}^{Z_A}})^{[1]} ( z) $\\
&&$ = \bigwedge_{(\beta, z')\in Z_A}[1(\beta, z')\to  (E(z', z) \to \beta)]$\\
&&$ = \bigwedge_{(\beta, z')\in Z_A}[E(z', z) \to \beta]$\\
&& $ =\beta$;\\
$\mathbb{M},  z\succ^\alpha \phi\vee \psi$ & iff & $\mathbb{M},  z\succ^\alpha \phi$ and $\mathbb{M},  z\succ^\alpha \psi$;\\
$\mathbb{M}, z\succ^\alpha \phi\wedge \psi$ & iff & $\alpha\leq (\val{\phi}\wedge\val{\psi})^{[1]}(z)$\\
&&$ = \bigwedge_{(\beta, z')\in Z_A}[(\val{\phi}\wedge\val{\psi})(\beta, z')$\\
&&$\quad \to  (E(z', z) \to \beta)]$;\\

$\mathbb{M}, z\succ^\alpha \Diamond \phi$ & iff & $\alpha\leq (R^{[0]}_\Diamond[\val{\phi}])(z)$\\
&& $ = \bigwedge_{ (\beta, z')\in Z_A}[\val{\phi}(\beta, z')$\\
&&$\quad \to (R_{\Diamond}(z, z') \to \beta)]$;\\

$\mathbb{M},  z\succ^\alpha \Box \phi$ & iff & $\alpha\leq ((R^{[0]}_\Box[\descr{\phi}])^{[1]})(z)$\\
&&$ = \bigwedge_{(\beta, z')\in Z_A}[R^{[0]}_\Box[\descr{\phi}](\beta, z')$\\ 
&&$\quad \to  (E(z', z) \to \beta)]$.\\
\end{tabular}
\end{center}
}}
\end{definition}

\begin{definition}\label{def:Sequent:True:In:Model}
	A sequent $\phi \vdash \psi$ is \emph{true in a model} $\mathbb{M} = (\mathbb{G}, V)$ (notation:  $\mathbb{M} \models \phi \vdash \psi$) if $\val{\phi}\subseteq \val{\psi}$, or equivalently, if $\descr{\psi}\subseteq \descr{\phi}$.
A sequent $\phi \vdash \psi$ is \emph{valid on a graph-based frame} $\mathbb{G}$ (notation:  $\mathbb{G} \models \phi \vdash \psi$) if $\phi \vdash \psi$ is true in every model $\mathbb{M} = (\mathbb{G}, V)$ based on $\mathbb{G}$.
\end{definition}
\begin{remark}\label{remark:Monotone:Vals}
It is not difficult to see that for all stable valuations, if $p\in \Prop$ and $\beta, \beta'\in \mathbf{A}$ such that $\beta\leq \beta'$, then $\val{p}(\beta, z) \leq \val{p}(\beta', z)$ for every $z\in Z$, and one can readily verify that this condition extends compositionally to every $\phi\in \mathcal{L}$. 
\end{remark}

Before moving on to the case study, let us expand on how to understand informally  the notion of $\alpha$-support at value-state pairs. To this end, it is perhaps useful to start analysing  $\mathbb{M}, (\beta, z)\Vdash^\alpha \bot$. By definition, this is the case iff  $ \alpha\leq\beta$. Hence, the role of $\beta$ in the pair $(\beta, z)$ is to indicate  the maximum extent $\alpha$ to which the `state' $(\beta, z)$  is allowed to $\alpha$-support a false statement. Equivalently, $(\beta, z)$ does not $\alpha$-support the falsehood for any $\alpha\not\leq \beta$. Hence, when $\mathbf{A}$ is linearly ordered,  $\beta$ indicates the `threshold'  beyond which (i.e.~overcoming which by going up) $\alpha$-support becomes meaningful at the given `state' $(\beta, z)$.  These observations open the way to the possibility of imposing extra conditions on the extension functions $\val{\phi}: \mathbf{A}\times Z\to \mathbf{A}$, depending on the given situation to be modelled. These extra conditions  are not required by the general semantic environment, but are accommodated by it.  For instance, if considering this threshold $\beta$ is not relevant to the given case at hand, then one can choose to restrict oneself to valuations such that, if $p\in \Prop$, then $\val{p}(\beta, z) = \val{p}(\beta', z)$ for every $z\in Z$ and every $\beta, \beta'\in \mathbf{A}$.  One can readily verify that this condition extends compositionally to every $\phi\in \mathcal{L}$.

\section{Case study: competing theories}\label{sec:Case:Study}
In the previous sections, we have illustrated how many-valued semantics for modal logic \cite{fitting1991many,bou2011minimum} can be generalized from  Kripke frames to graph-based structures. This generalization is  the parametric version  (where $\mathbf{A}$ is the parameter)  of the graph-based semantics of \cite{graph-based-wollic}, and  the main motivation for introducing it is that it allows for a rich description of certain essentials (in the present case, of the role of theories in the practice of empirical sciences), while still using basic intuitions from the crisp setting. For instance, in the many-valued setting, the basic intuition still holds that the generalization from classical Kripke frame semantics to graph-based semantics consists in thinking of the graph-relation $E$ as encoding an inherent boundary to knowability (referred to as {\em informational entropy}) which disappears in the classical setting (in which $E$ coincides with the identity relation).  Informational entropy can be due to many factors (e.g.~technological, theoretical, linguistic, perceptual, cognitive), and in \cite{graph-based-wollic}, examples are discussed in which the nature of  these limits is perceptual and linguistic.  In the present section, we discuss how  the {\em theoretical frameworks} adopted by empirical scientists can be  a source of informational entropy. 

For the purpose of this analysis, we consider graph-based structures   $(Z, E, \{R_{X_i}\mid X_i\subseteq \mathsf{Var}, 0\leq i\leq n\})$ in which $\mathbb{X} = (Z, E)$ represents a network of databases, and $\mathsf{Var}$ is a set of variables which includes the variables structuring the information contained in the databases of $Z$. In this context, $\mathcal{L}$-formulas can be thought of as hypotheses which will be assigned truth values (more specifically, truth-degrees) at  value-state pairs of models based on these frames. We will refer to any such  pair $(\beta, z)$ as a {\em situation},  the $\beta$ component of which is understood as the maximum degree of flexibility in operationalizing variables in that given situation.  This truth value assignment of a formula (hypothesis) at a value-state pair (situation)  is then intended to represent  the significance of the correlation posited by the hypothesis, when tested in the given database according to the degree of flexibility allowed at that situation,  with higher truth values   indicating higher levels of significance.\footnote{However, in this paper we do not intend to set up a systematic correlation between significance levels and truth values. The values chosen in the example below are only supposed to be intuitively plausible.} This interpretation is coherent with the property mentioned in Remark \ref{remark:Monotone:Vals}: indeed, the higher the flexibility in operationalizing variables, the more leeway to obtain a higher $\alpha$-support of hypotheses at situations.

Moreover, in the context of the graph-based structures above, an empirical theory is characterized by (and here identified with) a certain subset $X$ of variables which are {\em relevant} to the given theory; also, in what follows, for all databases $z_j\in Z$, we let $X_j$ denote the set of variables structuring the data contained in $z_j$.\footnote{It is interesting to notice that this basic environment naturally captures the idea that evidence is laden with theory: that is, we can think of each database $z_j$ as being constructed on the basis of the theory corresponding to set of variables $X_j$ associated with $z_j$.} Hence,
the  $\mathbf{A}$-relation $E$ encodes to what extent database $z_2$ is similar to $z_1$ (e.g.~by letting $E(z_1, z_2)$ record the percentage of variables of $z_1$ that also occur in $z_2$), while the relations $R_X$  encode to what extent one database is similar to another, relative to $X$ (e.g.~by letting $R_X(z_1, z_2)$ record the percentage of variables of $X_1\cap X$ that also occur in $X_2$). Below, we give a more concrete illustration of this environment by means of an example about {\em dietary theories}.  

The first theory, known since antiquity, is that body-fat loss of individuals depends on what they ate and how much exercise they did. We refer to it as the {\em ancient theory} ($A$),  on the basis of which Aretaeus the Cappadocian might have created a database $z_A$ recording how many {\em kochliaria} of olive oil and of honey, how many {\em minas} of bread, of olives, of lamb, and how many  {\em kyathoi} of wine a group of athletes and a group of rhetoric students ate each day, and how many  {\em stadia} they walked or ran each day, and how many {\em minas} each individual weighed each day. 

The second one, the {\em modern theory} ($M$), was developed in Victorian times by Wilbur Olin Atwater.\footnote{Source: \texttt{https://www.sciencehistory.org/\\
distillations/magazine/counting-calories}} In line with the Taylorist view on labour efficiency, the modern theory explains body-fat loss in terms of a negative balance between the daily caloric intake of individuals provided by food and their daily caloric expenditure, due e.g.~to maintaining body temperature or to exercise. The modern theory improves over the ancient in that it provides a common ground of commensurability, which was absent in the ancient theory,  between the variables relative to food intake and those relative to exercise, by reducing all of them to their energetic import, measured in calories.  Hence, an imaginary database $z_M$ built by Atwater on the basis of this theory would record how many calories individuals got from food and how many calories they spent per day, and their mass in kilograms measured each day.   


The third theory, referred to as the {\em hormonal response theory} ($H$),\footnote{Source: \texttt{https://idmprogram.com/\\
can-make-thininsulin-hormonal-obesity-v/}}
postulates that body-fat loss is governed by a hormone, insulin, which is released in response to the intake of certain macronutrients: namely, it is maximally released in response to intake of carbohydrates, less so but still significantly released in response to  protein intake, while fat intake does not trigger any significant insulin response.\footnote{Proviso:  for the purpose of keeping this example simple, we are  oversimplifying the hormonal response theory.} This theory posits that as long as insulin values are high, the body cannot access its own fat and use it as energy source,  no matter how severe the caloric restriction. An imaginary database $z_H$ built by  Banting and Best (who famously discovered insulin and its function) on the basis of this theory would  record how many calories {\em from carbohydrates}, how many calories {\em from proteins}, how many calories {\em from fat} individuals  got from food, how many calories they spent per day, and their mass in kilograms measured each day. 

This scenario can be modelled as the graph-based $\mathbf{A}$-frame $(Z, E, R_A, R_M, R_H)$, where $\mathbf{A}$ is the 11-element {\L}ukasiewicz chain,  $Z: = \{z_A, z_M, z_H\}$, and $E: Z\times Z\to \mathbf{A}$ is as indicated in the following diagram:

	
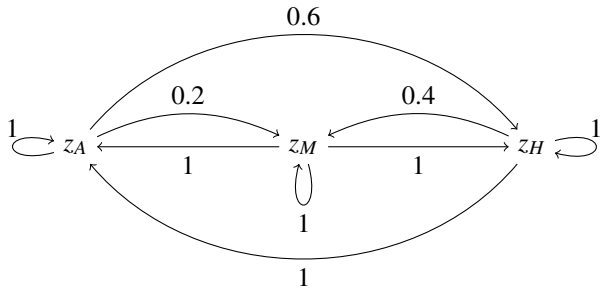
\begin{center}
		\begin{tikzpicture}[scale=1.5]
		\node (ancient)    at (-2,0)    {$z_A$};
		\node (modern)   at (0,0)    {$z_M$};
		\node (hormonal)    at (2,0)    {$z_H$};
		\draw[->] (modern) -- (ancient) node [midway,below]{1};
		\draw[->] (modern) -- (hormonal) node [midway,below]{1};
		\draw[->, bend left = 25]  (ancient) to node [midway,above]{$0.2$} (modern);
		\draw[->, bend right = 25]  (hormonal) to node [midway,above]{$0.4$} (modern);
		\draw[->,   loop below, min distance=5mm]  (modern) to node [midway,below]{$1$} (modern);
		\draw[->,  loop left, min distance=5mm]  (ancient) to node [midway,above]{$1$} (ancient);
		\draw[->,  loop right, min distance=5mm]  (hormonal) to node [midway,above]{$1$} (hormonal);
		\draw[->,  bend left=50] (ancient) to node [midway,above]{$0.6$} (hormonal);
		\draw[->,  bend left=50] (hormonal) to node [midway,below]{$1$} (ancient);		
		\end{tikzpicture}	
	\end{center} 
For any arrow in the diagram above, its value (i.e.~the similarity degree of the target to the source)  intuitively represents to which extent the target database provides information about variables relevant to the theory according to which the source database has been built. This similarity relation is reflexive by definition. Moreover, all arrows have non-zero values because  the three databases constructed on the basis of the three different theories have a minimal common ground, namely they all include the weight of individuals (which, as we will see, is the dependent variable in the hypotheses tested on them). 
We assign value 1 to all the arrows which have $z_M$ as source or $z_A$ as target, since the information relevant to the modern theory   can be fully retrieved from all databases, and   the information relevant to the modern and hormonal response theories can be fully retrieved from the variables in $z_A$.  The arrow from $z_A$ to $z_M$ is assigned the lowest non-zero value, since  from the information about the caloric intake and expenditure contained in $z_M$ one cannot retrieve the actual types of food the individuals ingested or the exercise they did.

For $X\in \{A, M, H\}$, and for any $z, z'\in Z$, the value of the $R_X$-arrow  from $z$ to $z'$  represents the similarity degree of $z'$ to $z$ {\em relative to} $X$. A concrete way to picture this is the following: assume that a scientist adopting theory $X$ is asked to which extent s/he would swap database $z$ for database $z'$. If the scientist in question is Aretaeus, and he is asked e.g.~to give up $z_H$ for $z_M$, he would not be very happy, for although $z_H$ is not particularly good for his purposes and requires  a substantial guesswork from him, he would be even worse off with $z_M$, and he would suffer the same loss of information captured by the value $E(z_H, z_M)$. This justifies letting $R_A(z_H, z_M): = E(z_H, z_M) = 0.4$. However, Aretaeus would certainly be willing to   swap $z_M$ for $z_H$, since whatever little he can do with $z_M$ can be certainly done with $z_H$, and in fact possibly more. Hence we let again $R_A(z_M, z_H): = E(z_M, z_H) = 1$, and so on. Hence, $R_A: = E$. If the scientist in question is Atwater, and he was asked to  give up $z_A$ for $z_H$, he would be fine with it, because both databases provide all the information relevant to the theoretical framework he has adopted. In fact, he would be fine with swapping any database for any other database: that is,  the relation $R_M: Z\times Z\to \mathbf{A}$ maps every tuple of databases to $1$.  An analogous reasoning  justifies the following definition for the relation $R_H: Z\times Z\to \mathbf{A}$:

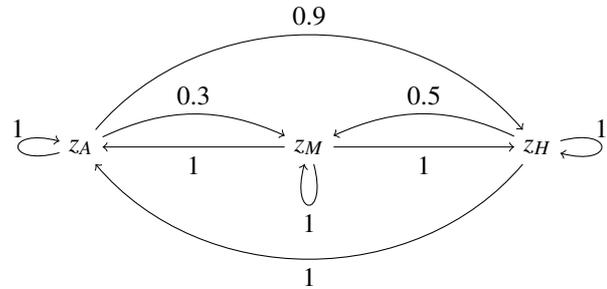
\begin{center}
	\begin{tikzpicture}[scale=1.5]
	\node (ancient)    at (-2,0)    {$z_A$};
	\node (modern)   at (0,0)    {$z_M$};
	\node (hormonal)    at (2,0)    {$z_H$};
	\draw[->] (modern) -- (ancient) node [midway,below]{1};
	\draw[->] (modern) -- (hormonal) node [midway,below]{1};
	\draw[->, bend left = 25]  (ancient) to node [midway,above]{$0.3$} (modern);
	\draw[->, bend right = 25]  (hormonal) to node [midway,above]{$0.5$} (modern);
	\draw[->,   loop below, min distance=5mm]  (modern) to node [midway,below]{$1$} (modern);
	\draw[->,  loop left, min distance=5mm]  (ancient) to node [midway,above]{$1$} (ancient);
	\draw[->,  loop right, min distance=5mm]  (hormonal) to node [midway,above]{$1$} (hormonal);
	\draw[->,  bend left=50] (ancient) to node [midway,above]{$0.9$} (hormonal);
	\draw[->,  bend left=50] (hormonal) to node [midway,below]{$1$} (ancient);		
	\end{tikzpicture}	
\end{center} 
where  $R_H$ and $E$ only differ in the value of the arrow from $z_A$ to $z_H$. The relations above are $E$-compatible (cf.~Definition \ref{def:graph:based:frame:and:model}), and $E$-reflexive (i.e.~$E\subseteq R$ for any $R\in \{R_A, R_M, R_H\}$, see \cite{graph-based-wollic} Definition 6) hence $(Z, E, R_A, R_M, R_H)$ is  a graph-based $\mathbf{A}$-frame for a multi-modal language with modalities $\Box_A, \Box_M, \Box_H$ and $\Diamond_A, \Diamond_M, \Diamond_H$, in which the axioms  $\Box_i \phi\vdash \phi$ and $ \phi\vdash \Diamond_i \phi$ are valid for every $i\in \{A, M, H\}$ (cf.~Proposition \ref{prop:correspondence}).\footnote{Notice that, although the setting of \cite{graph-based-wollic} is crisp, the correspondence results in \cite{graph-based-wollic} Proposition 4  remain verbatim the same when passing to the many-valued setting. This is a phenomenon already observed in the correspondence theory for many-valued logics \cite{CMR}.}

Let the formula $\phi$ be the hypothesis  stating that individuals who restrict their daily caloric intake to less than 20 calories per kilogram of body mass will lose weight over time. This hypothesis is phrased in terms of the variables relevant to the modern theory, and hence it can be tested on {\em all databases} in $Z$. Let us assume that the results of the tests of $\phi$ do not vary from one situation  to another situation with the same degree of flexibility $\beta$, and 
it turns out that, though 80\% of the individuals restricting their daily caloric intake to less than 20 calories per kilogram of body mass indeed lost a bit of weight, generally not too much, 10\% of individuals remained at the same weight, and 10\% even gained weight. 
Let us assume that in the statistical model this results in moderate effect size, but a p-value of 0.1, which is  considered to yield too low a level of significance to reject the null-hypothesis (that caloric restriction has no effect). So we propose, for the sake of this example, to assign  
$\val{\phi}: Z_A \to \mathbf{A}$ according to the following table\footnote{It can be checked that this valuation is stable.}:

\begin{center}
\begin{tabular}{r|lll}
	$\beta$ &$z_A$  &$z_M$  &$z_H$\\
	\hline
	$0.0$   &$0.5$  &$0.5$  &$0.5$\\
	$0.1$   &$0.6$  &$0.6$  &$0.6$\\
	$0.2$   &$0.7$  &$0.7$  &$0.7$\\
	$0.3$   &$0.8$  &$0.8$  &$0.8$\\
	$0.4$   &$0.9$  &$0.9$  &$0.9$\\
	$0.5$   &$1.0$  &$1.0$  &$1.0$\\
	$0.6$   &$1.0$  &$1.0$  &$1.0$\\
	$0.7$   &$1.0$  &$1.0$  &$1.0$\\
	$0.8$   &$1.0$  &$1.0$  &$1.0$\\
	$0.9$   &$1.0$  &$1.0$  &$1.0$\\
	$1.0$   &$1.0$  &$1.0$  &$1.0$
\end{tabular}
\end{center}


Let the formula $\psi$ be the hypothesis  stating that individuals who restrict their daily caloric intake to less than 20 calories per kilogram of body mass and who let  at least 80\% of their caloric intake come from fat will lose more weight than individuals on the same daily caloric regime but getting less than 80\% of their calories from fat. This hypothesis is phrased in terms of the variables relevant to the hormonal response theory, and hence it can certainly be tested on $z_A$ and $z_H$. We wish to make a case that, modulo some guesswork that of course will make the results less reliable, the hypothesis $\psi$  {\em can} be tested on $z_M$ as well. For instance, if the database $z_M$ built by Atwater was based on a group of individuals living in Connecticut in the years 1890-1895, and for some fortuitous circumstances an independent database exists about their eating habits (e.g.~the list  of  customers of the local grocer's and their weekly orders, and by chance these customers also include the people in $z_M$), then it would be possible to make some estimates about which individuals in the sample of $z_M$ let at least 80\% of their daily caloric intake come from fat. This is of course an easy way out in our fictitious example, but it reflects a very common situation in the practice of  empirical research, that databases do not perfectly match the hypotheses that scientists wish to test on them, and that some guesswork is needed to a greater or lesser extent. 

Notice that the imperfect match between the observations in databases and variables in hypotheses is independent from the (maximum) degree of flexibility in operationalising variables (formally encoded in the value $\beta$ of the pairs which we refer to as `situations'). Specifically, the degree of flexibility in operationalising any variables is an {\em a priori} parameter that we fix for each `situation', independently of the hypotheses tested in the given situation.  In contrast, the discussion in the paragraph above is relative to the test of a specific hypothesis on a database, and hence depends inherently on the given hypothesis. Furthermore, once such a suitable translation is found, its suitability will not depend on how the target variable is operationalised in each situation, but will depend only on the match between the theory according to which the database has been built and the theory to which the hypothesis pertains.

Let us imagine that 
$\psi$ is confirmed  for  95\% of the individuals in the  samples of all databases. 
Let us assume that in the statistical model this results in a high effect size for coefficient of the (dummy) variable recording whether the high-fat diet was followed or not and a p-value of 0.01, which corresponds to a level of significance generally considered to be high enough  to reject the null-hypothesis (that the type of macronutrients from which {\em restricted} caloric intake  proceeds has no effect on weight loss). In short, the results in respect to this hypothesis seem very strong and credible. So we propose, for $\beta = 0$, to assign a truth-value of  0.8 to $\psi$ at  $z_A$ and $z_H$, and  a truth-value of 0.4 to $\psi$ at $z_M$,  a strong discount due to the guesswork needed to accommodate the testing of $\psi$ on $z_M$.\footnote{For higher values of $\beta$, these values increase accordingly. It can be checked that the valuation as specified in the table, is stable.} The following table gives the complete specification of $\val{\psi}$:

\begin{center}
	\begin{tabular}{r|lll}
		$\beta$ &$z_A$  &$z_M$  &$z_H$\\
		\hline
		$0.0$   &$0.8$  &$0.4$  &$0.8$\\
		$0.1$   &$0.9$  &$0.5$  &$0.9$\\
		$0.2$   &$1.0$  &$0.6$  &$1.0$\\
		$0.3$   &$1.0$  &$0.7$  &$1.0$\\
		$0.4$   &$1.0$  &$0.8$  &$1.0$\\
		$0.5$   &$1.0$  &$0.9$  &$1.0$\\
		$0.6$   &$1.0$  &$1.0$  &$1.0$\\
		$0.7$   &$1.0$  &$1.0$  &$1.0$\\
		$0.8$   &$1.0$  &$1.0$  &$1.0$\\
		$0.9$   &$1.0$  &$1.0$  &$1.0$\\
		$1.0$   &$1.0$  &$1.0$  &$1.0$
	\end{tabular}
\end{center}

We are now in a position to compute the extensions of $\Box_M\phi$, $\Box_H\phi$, $\Box_M\psi$ and $\Box_H\psi$.\footnote{Since $R_A: = E$, the modal operators $\Box_A$ and $\Diamond_A$ coincide with the identity on $\mathbb{X}^+$.} Intuitively, $\Box_X\chi$ can be understood  as what becomes of  hypothesis $\chi$  when `seen through the lenses' of  theory $X$.\footnote{These modal operators can be used to reason about  ``comparative studies" which span across all databases and establish the degree of similarity between each databases and the focal one.}

It can be verified that:

\begin{align*}
&\val{\Box_M \psi}\\ 
= &\left[\begin{tabular}{r|lll}
$\beta$ &$z_A$  &$z_M$  &$z_H$\\
\hline
$0.0$   &$0.4$  &$0.4$  &$0.4$\\
$0.1$   &$0.5$  &$0.5$  &$0.5$\\
$0.2$   &$0.6$  &$0.6$  &$0.6$\\
$0.3$   &$0.7$  &$0.7$  &$0.7$\\
$0.4$   &$0.8$  &$0.8$  &$0.8$\\
$0.5$   &$0.9$  &$0.9$  &$0.9$\\
$0.6$   &$1.0$  &$1.0$  &$1.0$\\
$0.7$   &$1.0$  &$1.0$  &$1.0$\\
$0.8$   &$1.0$  &$1.0$  &$1.0$\\
$0.9$   &$1.0$  &$1.0$  &$1.0$\\
$1.0$   &$1.0$  &$1.0$  &$1.0$
\end{tabular}\right]
\leq &\left[\begin{tabular}{r|lll}
$\beta$ &$z_A$  &$z_M$  &$z_H$\\
\hline
$0.0$   &$0.5$  &$0.5$  &$0.5$\\
$0.1$   &$0.6$  &$0.6$  &$0.6$\\
$0.2$   &$0.7$  &$0.7$  &$0.7$\\
$0.3$   &$0.8$  &$0.8$  &$0.8$\\
$0.4$   &$0.9$  &$0.9$  &$0.9$\\
$0.5$   &$1.0$  &$1.0$  &$1.0$\\
$0.6$   &$1.0$  &$1.0$  &$1.0$\\
$0.7$   &$1.0$  &$1.0$  &$1.0$\\
$0.8$   &$1.0$  &$1.0$  &$1.0$\\
$0.9$   &$1.0$  &$1.0$  &$1.0$\\
$1.0$   &$1.0$  &$1.0$  &$1.0$
\end{tabular}\right]\\
 && = \val{\phi}
\end{align*}

and

\begin{align*}
&\val{\Box_M \psi}\\ 
= &\left[\begin{tabular}{r|lll}
$\beta$ &$z_A$  &$z_M$  &$z_H$\\
\hline
$0.0$   &$0.4$  &$0.4$  &$0.4$\\
$0.1$   &$0.5$  &$0.5$  &$0.5$\\
$0.2$   &$0.6$  &$0.6$  &$0.6$\\
$0.3$   &$0.7$  &$0.7$  &$0.7$\\
$0.4$   &$0.8$  &$0.8$  &$0.8$\\
$0.5$   &$0.9$  &$0.9$  &$0.9$\\
$0.6$   &$1.0$  &$1.0$  &$1.0$\\
$0.7$   &$1.0$  &$1.0$  &$1.0$\\
$0.8$   &$1.0$  &$1.0$  &$1.0$\\
$0.9$   &$1.0$  &$1.0$  &$1.0$\\
$1.0$   &$1.0$  &$1.0$  &$1.0$
\end{tabular}\right]
\leq &\left[\begin{tabular}{r|lll}
$\beta$ &$z_A$  &$z_M$  &$z_H$\\
\hline
$0.0$   &$0.8$  &$0.4$  &$0.8$\\
$0.1$   &$0.9$  &$0.5$  &$0.9$\\
$0.2$   &$1.0$  &$0.6$  &$1.0$\\
$0.3$   &$1.0$  &$0.7$  &$1.0$\\
$0.4$   &$1.0$  &$0.8$  &$1.0$\\
$0.5$   &$1.0$  &$0.9$  &$1.0$\\
$0.6$   &$1.0$  &$1.0$  &$1.0$\\
$0.7$   &$1.0$  &$1.0$  &$1.0$\\
$0.8$   &$1.0$  &$1.0$  &$1.0$\\
$0.9$   &$1.0$  &$1.0$  &$1.0$\\
$1.0$   &$1.0$  &$1.0$  &$1.0$
\end{tabular}\right]\\
&&= \val{\psi}
\end{align*}

It can also be checked that $\val{\Box_H \phi} = \val{\phi}$, $\val{\Box_H \psi} = \val{\psi}$ and $\val{\Box_M \phi} = \val{\phi}$.

These identities and inequalities can be interpreted as follows: each theory leaves unchanged the hypotheses formulated in terms of its own variables, or proper subsets thereof; however, if a hypothesis formulated according to a more expressive theory is `seen through the lenses' of a less expressive theory (this is the case of $\Box_M \psi$), it is expected to score worse. Finally, $\phi$ and $\psi$ are prima facie incomparable. But is it really so?

\section{Epilogue}\label{Sec:epilogue}
Although very stylised and simplified, the scenario above illuminates a number of interesting notions and their interrelations. First, we have identified each theory with the set of its relevant variables. This move naturally provides a connection with a strand of research we have been recently developing, based on the idea that lattice-based modal logics can be interpreted as the logics of categories or formal concepts \cite{conradie2016categories,Tarkpaper,roughconcepts}. This connection can be articulated in general terms by modelling {\em theories as categories}, extensionally captured by sets of hypotheses, and intensionally captured by their relevant variables.

Having identified theories with sets of variables has allowed us to associate states of the models (understood as databases) with theories, thereby giving a very simple and concrete representation of the otherwise abstract idea that `observations are theory-laden', and that this {\em theory-ladennes} lays at the core of the informational entropy that this paper sets out to studying.  In the toy example of the previous section, states (databases) bijectively correspond to theories, but this does not need to be the case in general.

Related to this, we have captured a local and a global way in which {\em similarity} ensues from theory-driven informational entropy. Specifically, the relation $E$ captures the {\em local} perspective, in which a given database $z'$ is similar to a database $z$ to the extent to which $z'$ is amenable to test hypotheses formulated using the variables of the theory associated with $z$, that is, to the extent to which $z'$ is suitable to answer questions pertaining to `the theory of $z$', while the  relations $R_X$ capture the {\em global} perspective; i.e., $R_X$ encodes information on how similar any one database is to another in respect to their relative performances in testing hypotheses formulated using variables of $X$.

These formal tools can be used to illuminate 
a very common situation in the practice of empirical research, namely that databases  do {\em not} perfectly match the hypotheses that scientists wish to test on them, and that a key underlying aspect in empirical research concerns precisely how to address this imperfect match. In this paper, we have laid the groundwork for addressing this issue with bespoke logical tools, by means of the modal operators interpreted using the relations $R_X$, which, as discussed above,  translate hypotheses from the `language' (variables) of one theory to the `language' (variables) of another, and what is lost in translation depends on the relationship between the two theories.

Finally, we can try and discuss whether the formalization above throws light on the following two questions: what does it mean for {\em theories} to {\em compete}? And how do we assess whether one theory has outcompeted the other?

We propose the following view:  theories can compete, most obviously when hypotheses that belong to different theories (as $\phi$ and $\psi$ in our example belong to $M$ and $H$ respectively) predict the same dependent variable (weight loss in our example). The most direct way in which two theories (e.g.~$M$ and $H$) can compete is when  their respective hypotheses ($\phi$ and $\psi$) are tested on each member of a set of databases. 
Each of these databases will be more or less suitable to test a given hypothesis. In particular, any hypothesis  is expected to get its best scores\footnote{We do not mean `best scores' in absolute terms, but in relative terms: that is, a bad hypothesis will score low on every database, but it will still get its higher scores on databases that are maximally compatible with the theory in the language of which the hypothesis is formulated.} if tested either on databases which are constructed in accordance with the theory in the variables of which the hypothesis is formulated, or on databases that are {\em maximally similar} to those.
We refer to all these databases as being  `home-ground' to that given hypothesis. For instance, in the case study of the previous section, every database  is `home-ground' to $\phi$, while only $z_A$ and $z_H$ are `home-ground' to $\psi$. When a hypothesis is tested on a database that is not its `home-ground', it will typically find no or less adequate values for its variables. The solution, as in our example of testing $\psi$ on the database $z_M$, is to look for proxies that represent to some extent the missing variable (or recover the values of the missing variables by some motivated guesswork). These proxies are often second-best or worse, making the results of the test less credible, even if they lead to accepting the hypotheses. This results in assigning the hypothesis a lower truth value at the database where proxies/guesswork were needed to test the hypothesis. 
However, precisely due to the disadvantage of not being on `home-ground', 
if a hypothesis pertaining to theory $X$ is tested on a database $z$ which is not `home-ground' to it 
and gets more than half as good results as the competing hypothesis, pertaining to theory $Y$, to which $z$ is `home-ground', this should be considered an impressive victory of theory $X$  and its hypothesis, just as is the effect of the rule that goals scored by soccer teams (in e.g.~the Champion’s League) in away matches count double. 


Applying this view to the  case of $\phi$ and $\psi$, we can hence argue that $\psi$ outcompetes $\phi$, since, as discussed above,  every database  is `home-ground' to $\phi$, while only $z_A$ and $z_H$ are `home-ground' to $\psi$, and moreover, $\psi$ scores systematically better than $\phi$ on each database that is `home-ground' to both and, even when $\beta = 0$,  the lower score of $\psi$ on $z_M$ (0.4) is very close to the score of $\phi$ on $z_M$ (0.5). 

\section{Conclusions}

In this paper, we have introduced a complete many-valued semantic environment for (multi) modal languages based on the logic of general  (i.e.~not necessarily distributive) lattices, and, by means of a toy example, we have illustrated its potential as a tool for the formal analysis of situations arising in the theory and practice of empirical science. As this is only a preliminary exploration, many questions arise, both technical and conceptual, of which here we list a few.

\paragraph{A range of protocols for comparing theories.} As discussed in Sections \ref{sec:Case:Study} and \ref{Sec:epilogue} hypotheses that look prima facie incomparable can become comparable through the lenses of the modal operators. In this paper we do not insist on a specific protocol to establish a winner among competing theories. However, we wish to highlight that the framework introduced can accommodate a wide range of possible protocols, including those  involving common-knowledge-type constructions.

\paragraph{More expressive languages.}  We conjecture that the proof of completeness of the logic of  Section \ref{sec:logics} given in Appendix \ref{sec:completeness} can be extended modularly to more expressive languages  that display essentially ``many valued" features in analogy with those considered in \cite{bou2011minimum}. This is current work in progress.

\paragraph{Sahlqvist theory for many-valued non-distributive logics.} A natural direction of research is to develop the generalized Sahlqvist theory for the logics of graph-based $\mathbf{A}$-frames, by extending the results of  \cite{CMR} on Sahlqvist theory for  many-valued logics on a distributive base.

\paragraph{Towards an analysis theory dynamics.}   We have shown that using the many-valued environment allows us to discuss competition between theories in an intuitively appealing and formally sound manner. This lays the basis for a host of further developments to model the dynamics of theories and databases, to better understand what distance between theories means, as well as studying the hierarchical relations between theories, in analogy with the hierarchical structure of categories.

\paragraph{Socio-political theories and scientific theories.} The present semantic environment naturally lends itself not only to the analysis of competition of scientific theories but also to the analysis of a wide spectrum of phenomena in which theories, broadly construed, play a key role. For instance, building on the present work, in \cite{socio-political}, a formal environment is introduced in which  the similarities can be analysed between the competition among political theories (both in their institutional incarnations as political parties, and in their social incarnations as social blocks or groups) and the competition between scientific theories.

\begin{acknowledgment}
	The authors would like to thank Apostolos Tzimoulis for insightful discussions and for his substantial contributions to the proof of the completeness theorem. 
\end{acknowledgment}

\bibliographystyle{eusflat2019}
\bibliography{BIBeusflat2019}

\appendix

\section{Correspondence results}
\label{sec:correspondence}
In what follows, for every graph-based $\mathcal{L}$-frame $\mathbb{G}$, we let $R_{\blacksquare}: Z\times Z\to\mathbf{A}$ be defined by the assignment $(z, z')\mapsto R_{\Diamond}(z', z)$. 
\begin{proposition}
\label{prop:correspondence}
The following are equivalent for every graph-based $\mathcal{L}$-frame $\mathbb{G}$:
\begin{enumerate}
\item $\mathbb{G}\models \Box\bot\vdash \bot$ iff for any $(\beta, z)\in Z_A$,
\[\bigwedge_{z'\in Z_X}[R_{\Box}(z, z') \to \beta] \leq \bigwedge_{z'\in Z_X}[E(z, z') \to \beta].\]
\item $\mathbb{G}\models \top\vdash \Diamond\top$ iff  for any $z\in Z_X$,
\[\bigwedge_{(\alpha, z')\in Z_A}[R_{\Diamond}(z, z') \to \alpha]\leq \bigwedge_{(\alpha, z')\in Z_A}[E(z', z) \to \alpha].\]
\item $\mathbb{G}\models \Box p\vdash p\quad$ iff $\quad E\subseteq R_{\Box}$.
\item $\mathbb{G}\models p\vdash \Diamond p\quad $ iff $\quad E\subseteq R_{\blacksquare}$.
\end{enumerate}
\end{proposition}
\begin{proof}
1. 
\begin{center}
\begin{tabular}{c l}
& $\Box\bot\leq \bot$\\
iff & $R^{[0]}_{\Box}[\descr{\bot}]\leq \val{\bot}$\\
iff & $R^{[0]}_{\Box}[1^{\mathbf{A}^{Z_X}}]\leq (1^{\mathbf{A}^{Z_X}})^{[0]}$\\
iff & $R^{[0]}_{\Box}[1^{\mathbf{A}^{Z_X}}](\beta, z)\leq (1^{\mathbf{A}^{Z_X}})^{[0]}(\beta, z)$  $\forall(\beta, z)\in Z_A$\\
iff & $\bigwedge_{z'\in Z_X}[1(z')\to (R_{\Box}(z, z') \to \beta)]$\\
& $\leq \bigwedge_{ z'\in Z_X}[1(z')\to (E(z, z') \to \beta)]$   $\forall(\beta, z)\in Z_A$\\
iff & $\bigwedge_{z'\in Z_X}[R_{\Box}(z, z') \to \beta]$\\
& $\leq \bigwedge_{z'\in Z_X}[E(z, z') \to \beta]$  for any $(\beta, z)\in Z_A$.\\
\end{tabular}
\end{center}

2. 
\begin{center}
\begin{tabular}{c l}
& $\top\leq \Diamond\top$\\
iff & $R^{[0]}_{\Diamond}[\val{\top}]\leq \descr{\top}$\\
iff & $R^{[0]}_{\Diamond}[1^{\mathbf{A}^{Z_A}}]\leq (1^{\mathbf{A}^{Z_A}})^{[1]}$\\
iff & $R^{[0]}_{\Diamond}[1^{\mathbf{A}^{Z_A}}](z)\leq (1^{\mathbf{A}^{Z_A}})^{[0]}( z)$ for any $z\in Z_X$\\
iff & $\bigwedge_{(\alpha, z')\in Z_A}[1(\alpha, z')\to (R_{\Diamond}(z, z') \to \alpha)]$\\
& $\leq \bigwedge_{(\alpha, z')'\in Z_A}[1(\alpha,z')\to (E(z', z) \to \alpha)]$   $\forall z\in Z_X$\\
iff & $\bigwedge_{(\alpha, z')\in Z_A}[R_{\Diamond}(z, z') \to \alpha]$\\
& $\leq \bigwedge_{(\alpha, z')\in Z_A}[E(z', z) \to \alpha]$  for any $z\in Z_X$.\\
\end{tabular}
\end{center}

3. 
	\begin{center}
		\begin{tabular}{cll}
			&$\forall p [\Box p \leq p]$\\
			iff &$\forall \cnomm [\Box \cnomm \leq \cnomm]$ &(ALBA \cite{CoPa:non-dist})\\
			iff &$\forall \alpha \forall z [R^{[0]}_{\Box}[\{\alpha/z\}^{[01]}]  \leq \{\alpha/z\}^{[0]}]$ & ($\cnomm := \{\alpha/z\}$) \\
			iff &$\forall \alpha \forall z [R^{[0]}_{\Box}[\{\alpha/z\}]  \leq \{\alpha/z\}^{[0]}]$ &(Lemma \ref{equivalence of I-compatible-mv})\\
			iff &$\forall \alpha, \beta \forall z, w[\alpha \to (R_{\Box}(w, z) \to\beta) $\\
			& $\leq \alpha \to (E(w, z) \to\beta)$ &($\ast$)\\
			iff &$E\subseteq R_{\Box}$. &($\ast\ast$)
		\end{tabular}
	\end{center} 
	To justify the equivalence to ($\ast$) we note that $R^{[0]}_{\Box}[\{\alpha/z\}](\beta, w) = \bigwedge_{z' \in Z_X}[\{\alpha/z\}(z') \to (R_{\Box}(w, z')\to\beta)]  = \alpha \to (R_{\Box}(w, z) \to\beta)$, and moreover $\{\alpha/z\}^{[0]}(\beta, w) = \bigwedge_{z' \in Z_X}(\{\alpha/z\}(z') \to (E(w, z')\to\beta)) = \alpha \to (E(w, z)\to\beta)$. For the equivalence to ($\ast\ast$), note that instantiating $\alpha: = 1$ and $\beta: = R_{\Box} (w, z)$ in ($\ast$) yields $1 \leq E(w, z) \to R_{\Box} (w, z)$ which, by residuation, is equivalent to ($\ast\ast$). The converse direction is immediate by the monotonicity of $\to$ in the second coordinate and its antitonicity in the first coordinate.	
	
4.
\begin{center}
\begin{tabular}{c ll}
& $p\leq \Diamond p$\\
iff &$\forall \nomj[\nomj\leq \Diamond \nomj]$ & \\ 
	iff &$\forall \alpha,\beta \forall z [R^{[0]}_{\Diamond}[\{\alpha/(\beta, z)\}^{[01]}]  \leq \{\alpha/(\beta, z)\}^{[1]}]$ & \\
iff &$\forall \alpha, \beta \forall z [R^{[0]}_{\Diamond}[\{\alpha/(\beta, z)\}]  \leq \{\alpha/(\beta, z)\}^{[1]}]$ & \\ 
iff &$\forall \alpha, \beta \forall z, w[\alpha \to (R_{\Diamond}(w, z)\to\beta) $\\
			& $\leq \alpha \to (E(z, w) \to\beta)$ &($\ast$)\\
			iff &$E\subseteq R_{\blacksquare}$. &($\ast\ast$)
\end{tabular}
\end{center}
	As to the equivalence to ($\ast$), note that $R^{[0]}_{\Diamond}[\{\alpha/(\beta, z)\}] (w) = \bigwedge_{(\gamma, z') \in Z_A}[\{\alpha/(\beta, z)\}(\gamma, z') \to (R_{\Diamond}(w, z')\to\gamma)]  = \alpha \to (R_{\Diamond}(w, z) \to\beta)$, and that $\{\alpha/(\beta, z)\}^{[1]}(w) = \bigwedge_{(\gamma, z') \in Z_A}(\{\alpha/(\beta, z)\}(\gamma, z') \to (E(z',w)\to\gamma)) = \alpha \to (E(z, w)\to\beta)$. For the equivalence to ($\ast\ast$) note that instantiating $\alpha := 1$ and $\beta: = R_{\Diamond} (w, z) = R_{\blacksquare}(z, w)$ in ($\ast$) yields $1 \leq E(z, w) \to R_{\blacksquare}(z, w)$ which, by residuation, is equivalent to ($\ast\ast$). The converse direction is immediate by the monotonicity of $\to$ in the second coordinate and its antitonicity in the first coordinate.	
\end{proof}

\begin{remark}
Clearly, $E$-reflexivity (i.e.~condition $E\subseteq R_{\Box}$ in Proposition \ref{prop:correspondence}.3) implies the inequality in Proposition \ref{prop:correspondence}.1; however, this inequality is also verified under weaker but practically relevant assumptions. For instance,
if $\mathbf{A}$ is a finite chain, the inequality in Proposition \ref{prop:correspondence}.1 is equivalent to $\mathrm{min}\{R_{\Box}(z, z') \to \beta\mid z'\in Z_X\}\leq \mathrm{min}\{E(z, z') \to \beta\mid z'\in Z_X\} = E(z, z)\to\beta = 1\to\beta = \beta$. Hence, this condition is equivalent to the condition that for every $\beta\in \mathbf{A}$ and $z\in Z$, some $z'\in Z$ exists such that    $R_{\Box}(z, z') \to \beta\leq \beta$. This condition is satisfied if for every $z\in Z$ some $z'\in Z$ exists such that    $R_{\Box}(z, z')  = 1$. Similar considerations apply to the remaining items of the proposition above.
\end{remark}

\section{Completeness}
\label{sec:completeness}
\newcommand{\PP}{\mathbb{Fm}}
\newcommand{\SD}{\mathbb{Fm}}
\newcommand{\SI}{\mathbb{Fm}}

This section is an adaptation and expansion of  the completeness result of \cite[Appendix B]{socio-political}, of which Apostolos Tzimoulis and Claudette Robinson are prime contributors. We will use the validity of $\Box \bot\vdash \bot$  in the proof of the $\Box$ case in Lemma \ref{lemma:truthlemma}. As discussed in Section \ref{sec:Case:Study}, this axiom is valid in the model of our case study. 

For the sake of uniformity with previous settings (cf.~e.g.~\cite[Section 7.2]{roughconcepts}) in this section, we work with graph-based frames $\mathbb{G} = (\mathbb{X}, R_{\Box}, R_{\Diamond})$ the associated complex algebras of which are order-dual to the one in Definition \ref{def:graph:based:frame:and:model}. That is, for the sake of this section, we define the enriched formal context $\mathbb{P}_{\mathbb{G}}: = (Z_A, Z_X, I_{E} , I_{R_\Box}, J_{R_\Diamond})$ by setting $Z_A: = Z$, $Z_X: = \mathbf{A}\times Z$ and $I_{E}: Z_A\times Z_X\to \mathbf{A}$ and
$I_{R_\Box}: Z_A\times Z_X\to \mathbf{A}$ and $J_{R_\Diamond}: Z_X\times Z_A\to \mathbf{A}$ be defined by the assignments $(z, (\alpha, z'))\mapsto E(z, z')\to \alpha$,
 $(z, (\alpha, z'))\mapsto R_{\Box}(z, z')\to \alpha$ and $((\alpha, z), z')\mapsto R_{\Diamond}(z, z')\to \alpha$, respectively.

For any lattice $\mathbb{L}$, an $\mathbf{A}$-{\em filter} is an $\mathbf{A}$-subset of $\mathbb{L}$, i.e.~a map $f:\mathbb{L}\to \mathbf{A}$, which is both $\wedge$- and $\top$-preserving, i.e.~$f(\top) = 1$ and $f(a\wedge b) = f(a)\wedge f(b)$ for any $a, b\in\mathbb{L}$. Intuitively, the $\wedge$-preservation encodes a many-valued version of closure under $\wedge$ of filters. An $\mathbf{A}$-filter is {\em proper} if it is also $\bot$-preserving, i.e.~$f(\bot)= 0$.   Dually, an $\mathbf{A}$-{\em ideal} is a map $i:\mathbb{L}\to \mathbf{A}$ which is both $\vee$- and $\bot$-reversing, i.e.~$i(\bot) = \top$ and $i(a\vee b) = i(a)\wedge i(b)$ for any $a, b\in\mathbb{L}$, and is {\em proper} if in addition $i(\top) = 0$. The {\em complement} of a (proper) $\mathbf{A}$-ideal is a map $u: \mathbb{L}\to \mathbf{A}$  which is both $\vee$- and $\bot$-preserving, i.e.~$u(\bot) = 0$ and $u(a\vee b) = u(a)\vee u(b)$ for any $a, b\in\mathbb{L}$ (and in addition $u(\top) = 1$). Intuitively, $u(a)$ encodes the extent to which $a$ does not belong to the ideal of which $u$ is the many-valued complement.  We let $\mathsf{F}_{\mathbf{A}}(\mathbb{L})$, $\mathsf{I}_{\mathbf{A}}(\mathbb{L})$ and $\mathsf{C}_{\mathbf{A}}(\mathbb{L})$ respectively denote the set of  proper $\mathbf{A}$-filters,   proper $\mathbf{A}$-ideals, and  the complements of proper $\mathbf{A}$-ideals of $\mathbb{L}$. 
For any $\mathcal{L}$-algebra $(\mathbb{L}, \Box, \Diamond)$, and any $\mathbf{A}$-subset $k:\mathbb{L}\to \mathbf{A}$, let $k^{-\Diamond}: \mathbb{L}\to \mathbf{A}$  be defined as $k^{-\Diamond}(a)=\bigvee\{k(b)\mid \Diamond b\leq a\}$
and let $k^{-\Box}: \mathbb{L}\to \mathbf{A}$  be defined as $k^{-\Box}(a)=\bigwedge \{k(b)\mid a \leq \Box b \}$. 
By definition one can see that $k(a)\leq k^{-\Diamond}(\Diamond a)$ and 
$k^{-\Box}(\Box a) \leq k(a)$ for every $a\in \mathbb{L}$.
Let $\mathbf{Fm}$ (resp.\ $\mathbf{Fm}_0$, $\mathbf{Fm}_1$) be the Lindenbaum-Tarski algebra of the basic $\mathcal{L}$-logic $\mathbf{L}$ (resp.\ $\mathbf{L}_0$, $\mathbf{L}_1$). Moreover, in what follows we write $\mathbf{Fm}_{\ast}$ for the  Lindenbaum-Tarski algebra of an arbitrary (not necessarily proper) extension $\mathbf{L}_{\ast}$ of $\mathbf{L}$. In the remainder of this section, we abuse notation and identify formulas with their equivalence class in $\mathbf{Fm}_{\ast}$. 
\begin{lemma}\label{lemma:f minus diam is filter}
	\hspace{2em}
\begin{enumerate}
\item If $f:\mathbb{L}\to \mathbf{A}$ is  an $\mathbf{A}$-filter, then so is $f^{-\Diamond}$.  
\item If $f:\mathbf{Fm}_{\ast}\to \mathbf{A}$ is  a proper $\mathbf{A}$-filter, then so is $f^{-\Diamond}$.
\item If $u:\mathbb{L}\to \mathbf{A}$ is  the complement of an $\mathbf{A}$-ideal, then so is $u^{-\Box}$.  
\item If $u:\mathbf{Fm}_{\ast}\to \mathbf{A}$ is  the complement of a proper $\mathbf{A}$-ideal, then so is $u^{-\Box}$.
\item If $\phi, \psi\in \mathbf{Fm}_{\ast}$, then  $\phi\vee\psi = \top$ implies that  $\phi = \top$ or $\psi = \top$.
\item If $\phi, \psi\in \mathbf{Fm}_{\ast}$, then     $\phi \not\vdash \bot$ and $\psi \not\vdash \bot$ imply that $\phi\wedge\psi \not\vdash \bot$.
\end{enumerate}
\end{lemma}  
\begin{proof} 
1.  
For all $a, b\in \mathbb{L}$, 
\begin{center}
\begin{tabular}{cll}
   $f^{-\Diamond}(\top)$ & = & $\bigvee\{f(b)\mid \Diamond b\leq \top\}$\\
   & = & $\bigvee\{f(b)\mid b\in \mathbb{L}\}$\\
   & = & $f(\top)$\\
   & = & $1$\\
\end{tabular}
\end{center}
\begin{center}
\begin{tabular}{cll}
   &$f^{-\Diamond}(a)\wedge f^{-\Diamond}( b)$\\
    = &$ \bigvee\{f(c_1)\mid \Diamond c_1\leq a\} \wedge \bigvee\{f(c_2)\mid \Diamond c_2\leq b\}$\\
        = &$ \bigvee\{f(c_1)\wedge f(c_2)\mid \Diamond c_1\leq a\mbox{ and }\Diamond c_2\leq b\}$ & ($\star$) \\ 
             = &$ \bigvee\{f(c_1\wedge c_2)\mid \Diamond c_1\leq a\mbox{ and }\Diamond c_2\leq b\}$ &($\sharp$) \\
               = &$ \bigvee\{f(c)\mid \Diamond c\leq a\mbox{ and }\Diamond c\leq b\}$ & ($\ast$)\\
 = &$\bigvee\{f(c)\mid \Diamond c\leq a\wedge b\}$ & \\
=&$f^{-\Diamond}(a\wedge b)$,\\
 \end{tabular}
 \end{center}
 the equivalence marked with ($\star$) being due to frame distributivity, the one marked with ($\sharp$) to the fact that $f$ is and $\mathbf{A}$-filter, and the one marked with ($\ast$) to the fact that $\Diamond(c_1\wedge c_2)\leq \Diamond c_1\wedge \Diamond c_2$.

2. In general, $f^{-\Diamond}$ need not be a \emph{proper} filter, even if $f$ is. However, let us show that this is the case when $f$ is a proper filter of  $\mathbf{Fm}$. Indeed, in this algebra, $f^{-\Diamond}(\bot) = \bigvee\{f([\phi])\mid [\Diamond \phi] \leq [\bot]\} = \bigvee\{f([\phi])\mid \Diamond \phi \vdash \bot \} = \bigvee\{f([\phi])\mid \phi \vdash \bot \} = f([\bot]) = 0$. The crucial inequality is the third to last, which holds since  $\Diamond \phi \vdash \bot$ iff $\phi \vdash \bot$. The right to left implication can be easily derived in 
$\mathbf{L}$. 
For the sake of the left to right implication we appeal to the completeness of 
$\mathbf{L}$ 
with respect to the class of all normal lattice expansions of the appropriate signature \cite{CoPa:non-dist} and reason contrapositively. Suppose $\phi \not\vdash \bot$. Then, by this completeness theorem, there is a lattice expansion $\mathbb{C}$ and assignment $v$ on $\mathbb{C}$ such that $v(\phi) \neq 0$. Now consider the algebra $\mathbb{C}'$ obtained from $\mathbb{C}$ by adding a new least element $0'$ and extending the $\Diamond$-operation by declaring $\Diamond 0' = 0'$. We keep the assignment $v$ unchanged. It is easy to check that $\mathbb{C}'$ is a normal lattice expansion, and that $v(\Diamond \phi) \geq 0 > 0'$ and hence $\Diamond \phi \not \vdash \bot$. 

Items 3 and 4 are proven by  arguments which are dual to the ones above.

5. 
%
As to proving item 5,  we reason contrapositively. Suppose $\top \not\vdash \phi$ and $\top \not\vdash \psi$. 
By the completeness theorem to which we have appealed in the proof of item 2, there are lattice expansions $\mathbb{C}_1$ and $\mathbb{C}_2$ and corresponding assignments $v_i$ on $\mathbb{C}_i$ such that $v_1(\phi) \neq \top^{\mathbb{C}_1}$ and $v_2(\psi) \neq \top^{\mathbb{C}_2}$. Consider the  algebra $\mathbb{C}'$ obtained  by adding a new top element $\top'$ to  $\mathbb{C}_1\times \mathbb{C}_2$, defining the operation $\Diamond' = \Diamond^{\mathbb{C}'}$ by the same assignment of $\Diamond^{\mathbb{C}_1\times \mathbb{C}_2}$ on $\mathbb{C}_1\times \mathbb{C}_2$ and mapping $\top'$ to $(\Diamond \top)^{\mathbb{C}_1\times \mathbb{C}_2}$, and the operation $\Box' = \Box^{\mathbb{C}'}$ by the same assignment of $\Box^{\mathbb{C}_1\times \mathbb{C}_2}$ on $\mathbb{C}_1\times \mathbb{C}_2$, and mapping $\top'$  to itself. The normality (i.e.~finite meet-preservation) of $\Box'$ and the monotonicity of $\Diamond'$ follow immediately by construction. The normality (i.e.~finite join-preservation) of $\Diamond'$ is verified by cases: if $a\lor b\neq \top'$, then it immediately follows from the normality of $\Diamond^{\mathbb{C}_1\times\mathbb{C}_2}$.  If $a\lor b=\top'$, then by construction, either $a=\top'$ or $b=\top'$ (i.e.~$\top'$ is join-irreducible), and hence, the join-preservation of $\Diamond'$ is a consequence of its monotonicity. 
Consider  the valuation $v': \mathsf{Prop}\to \mathbb{C}'$ defined by the assignment $p\mapsto e (v_1(p), v_2(p))$, where $e: \mathbb{C}_1\times \mathbb{C}_1\to \mathbb{C}'$ is the natural embedding. 


Let us show, for all $\chi \in \mathcal{L}$, that if $(v_1(\chi), v_2(\chi)) \neq \top^{\mathbb{C}_1 \times \mathbb{C}_2}$, then $v'(\chi) \neq \top'$. We proceed by induction on $\chi$. The cases for atomic propositions and conjunction are immediate. The case for disjunction uses the join-irreducibility of $\top'$. When $\chi := \Diamond \theta$, then $v'(\Diamond \theta) = \Diamond' v'(\theta)\neq \top'$, since,  by construction, $\top'$ is not in the range of $\Diamond'$.

If $\chi: = \Box \theta$, then $v'(\chi) = v'(\Box \theta) = \Box' v'(\theta)$. Then the assumption that $(v_1(\chi), v_2(\chi)) \neq \top^{\mathbb{C}_1 \times \mathbb{C}_2}$ implies that $v'(\theta) \neq \top'$. Indeed,  if $v'(\theta) = \top'$, then, by induction hypothesis, $(v_1(\theta), v_2(\theta)) = (\top^{\mathbb{C}_1}, \top^{\mathbb{C}_2})$ and hence $(v_1(\Box \theta), v_2(\Box \theta)) = \top^{\mathbb{C}_1 \times \mathbb{C}_2}$. Therefore, from $v'(\theta) \neq \top'$, it follows from the definition of $\Box'$ that $v'(\Box \theta) = \Box' v'(\theta) \neq \top'$, which concludes the proof of the claim.

Clearly, $v_1(\phi) \neq \top^{\mathbb{C}_1}$ and $v_2(\psi) \neq \top^{\mathbb{C}_2}$ imply that $(v_1(\phi), v_2(\phi)) \neq \top^{\mathbb{C}_1 \times \mathbb{C}_2}$ and $(v_1(\psi), v_2(\psi)) \neq \top^{\mathbb{C}_1 \times \mathbb{C}_2}$. So, by the above claim,  $v'(\phi) \neq \top'$ and $v'(\psi) \neq \top'$, and hence, since $\top'$ is join-irreducible, $v'(\phi \vee \psi) \neq \top'$.

The proof of item 6 is dual to the one above.
 \end{proof}


\begin{lemma}\label{eq:premagicnew2} 
For any $f\in \mathsf{F}_{\mathbf{A}}(\mathbb{L})$ and any $u\in \mathsf{C}_{\mathbf{A}}(\mathbb{L})$,
\begin{enumerate}
\item $\bigwedge_{b\in \mathbb{L}} (f^{-\Diamond}(b)\to u(b)) = \bigwedge_{a\in \mathbb{L}} (f(a)\to u(\Diamond a))$;
\item $\bigwedge_{b\in \mathbb{L}} (f(b) \to u^{-\Box}(b)) = \bigwedge_{a\in \mathbb{L}} (f(\Box a)\to u(a))$. 
\end{enumerate}
\end{lemma}
\begin{proof} For (1) we use the fact that
$f(a)\leq f^{-\Diamond}(\Diamond a)$ implies that  
$f^{-\Diamond}(\Diamond a)\to u(\Diamond a)\leq f(a)\to u(\Diamond a)$ for every $a\in \mathbb{L}$, which is enough to show that $\bigwedge_{b\in \mathbb{L}} (f^{-\Diamond}(b)\to u(b)) \leq \bigwedge_{a\in \mathbb{L}} (f(a)\to u(\Diamond a)).$ Conversely, to show that \[\bigwedge_{a\in \mathbb{L}} (f(a)\to u(\Diamond a))\leq \bigwedge_{b\in \mathbb{L}} (f^{-\Diamond}(b)\to u(b)),\] 
it is enough to show that, for every $b\in \mathbb{L}$,
 \[\bigwedge_{a\in \mathbb{L}} (f(a)\to u(\Diamond a))\leq f^{-\Diamond}(b)\to u(b),\] 
 i.e.~by definition of $f^{-\Diamond}(b)$ and the fact that $\to$ is completely join-reversing in its first coordinate,
  \[\bigwedge_{a\in \mathbb{L}} (f(a)\to u(\Diamond a))\leq \bigwedge_{\Diamond c\leq b} (f(c)\to u(b)).\] 
  Hence, let $c\in \mathbb{L}$ such that $\Diamond c\leq b$, and let us show that   \[\bigwedge_{a\in \mathbb{L}} (f(a)\to u(\Diamond a))\leq f(c)\to u(b).\] 
 Since $u$ is $\vee$-preserving, hence order-preserving, $\Diamond c\leq b$ implies $u(\Diamond c)\leq u(b)$, hence
 \[\bigwedge_{a\in \mathbb{L}} (f(a)\to u(\Diamond a))\leq f(c)\to u(\Diamond c) \leq f(c)\to u(b),\]
 as required. 
For (2), we use $u^{-\Box}(\Box a) \leq u(a)$ and the fact that $\to$ is order-preserving in the second coordinate to show the inequality 
$\bigwedge_{b\in \mathbb{L}} (f(b) \to u^{-\Box}(b)) \leq \bigwedge_{a\in \mathbb{L}} (f(\Box a)\to u(a))$. 
To show 
\[\bigwedge_{a\in \mathbb{L}} (f(\Box a)\to u(a)) \leq \bigwedge_{b\in \mathbb{L}} (f(b) \to u^{-\Box}(b))\] 
we can show that for any $b \in \mathbb{L}$ 
\[
\bigwedge_{a\in \mathbb{L}} (f(\Box a)\to u(a)) \leq f(b) \to u^{-\Box}(b). 
\]   
After applying the definition of $u^{-\Box}(b)$ and the fact that $\to$ is completely meet-preserving in its second coordinate, the above inequality is equivalent to 
\[
\bigwedge_{a\in \mathbb{L}} (f(\Box a)\to u(a)) \leq \bigwedge_{b \leq \Box c} ( f(b) \to u(c)). 
\]
Let $c \in \mathbb{L}$ with $b \leq \Box c$. Since $f$ is order-preserving we get 
\[ 
\bigwedge_{a\in \mathbb{L}} (f(\Box a)\to u(a)) \leq f(\Box c) \to u(c) \leq f(b) \to u(c). 
\]
\end{proof}
\begin{definition}
\label{def:canonical frame}
The {\em canonical} graph-based $\mathbf{A}$-{\em frame} associated with any $\mathbf{Fm}_{\ast}$ is the structure  $\mathbb{G}_{\mathbf{Fm}_{\ast}} = (Z, E, R_{\Diamond}, R_\Box)$  defined as follows:\footnote{Recall that for any set $W$, the $\mathbf{A}$-{\em subsethood} relation between elements of $\mathbf{A}$-subsets of $W$ is the map $S_W:\mathbf{A}^W\times \mathbf{A}^W\to \mathbf{A}$ defined as $S_W(f, g) :=\bigwedge_{w\in W}(f(w)\rightarrow g(w)) $. 
If $S_W(f, g) =1$ we also write $f\subseteq g$. 
}	
	{\small 
	\[Z:=\{(f,u)\in \mathsf{F}_{\mathbf{A}}(\mathbf{Fm}_{\ast})\times \mathsf{C}_{\mathbf{A}}(\mathbf{Fm}_{\ast})\mid \bigwedge_{\phi\in \mathbf{Fm}_{\ast}}(f(\phi)\to u(\phi))=1\}.\]
	}
		
	For any $z\in Z$ as above, we let $f_z$ and $u_z$ denote the first and the second coordinate of $z$, respectively. 
Then $E: Z\times Z\to \mathbf{A}$, $R_\Diamond :Z\times Z\to \mathbf{A}$ and $R_{\Box}: Z\times Z \to \mathbf{A}$ 
	are  defined as follows: \[E(z, z'): = \bigwedge_{\phi\in \mathbf{Fm_{\ast}}}(f_z(\phi)\to u_{z'}(\phi));\]
	{\small 
	 \[R_\Diamond(z, z') := \bigwedge_{\phi\in \mathbf{Fm}_{\ast}}(f_{z'}^{-\Diamond}(\phi)\to  u_z(\phi)) = \bigwedge_{\phi\in \mathbf{Fm}_{\ast}}(f_{z'}(\phi)\to  u_z(\Diamond\phi));\]
	}
{\small
\[
R_{\Box}(z,z'):= \bigwedge_{\phi \in \mathbf{Fm}_{\ast}} (f_z(\phi) \to u_{z'}^{-\Box}(\phi)) = \bigwedge_{\phi \in \mathbf{Fm}_{\ast}}
(f_z(\Box \phi)\to u_{z'}(\phi)).
\]
}
We will write $\mathbb{G} = (Z, E, R_{\Diamond}, R_\Box)$ for $\mathbb{G}_{\mathbf{Fm}_{\ast}} = (Z, E, R_{\Diamond}, R_\Box)$ whenever ${\mathbf{Fm}_{\ast}}$ is clear from the context.
	\end{definition}
	\begin{lemma}
	\label{lemma:compatibility}
	The structure $\mathbb{G}_{\mathbf{Fm}_{\ast}}$ of Definition \ref{def:canonical frame}
 is a graph-based $\mathbf{A}$-frame, in the sense specified at the beginning of the present section.
	\end{lemma}
	\begin{proof}
We need to show that $R_{\Diamond}$ is $E$-compatible, i.e.,
\begin{align*}
(R_\Diamond^{[1]}[\{\beta / (\alpha, z) \}])^{[10]} &\subseteq R_\Diamond^{[1]}[\{\beta /  (\alpha, z) \}] \\
(R_\Diamond^{[0]}[\{\beta /  z \}])^{[01]} &\subseteq R_\Diamond^{[0]}[\{\beta / z \}],
\end{align*}
and that $R_{\Box}$ is $E$-compatible, i.e.,
\begin{align*}
(R_\Box^{[0]}[\{\beta /  (\alpha, z) \}])^{[10]} &\subseteq R_\Box^{[0]}[\{\beta / (\alpha, z) \}] \\
(R_\Box^{[1]}[\{\beta /  z \}])^{[01]} &\subseteq R_\Box^{[1]}[\{\beta /  z \}].
\end{align*}
Considering the second inclusion for $R_{\Diamond}$, by definition, for any $(\alpha, w)\in Z_X$,
\begin{center}
\begin{tabular}{c l}
&$R_{\Diamond}^{[0]}[\{\beta / z \}](\alpha, w)$ \\
 = & $\bigwedge_{z'\in Z_A}[\{\beta / z \}(z')\to  (R_{\Diamond}(w, z') \to \alpha)]$\\
  = & $\beta\to (R_{\Diamond}(w, z) \to \alpha)$\\
  & \\
 & $(R_\Diamond^{[0]}[\{\beta /  z \}])^{[01]} (\alpha, w)$\\
 = & $\bigwedge_{z'\in Z_A}[(R_\Diamond^{[0]}[\{\beta /  z \}])^{[0]}(z')\to (E(z', w)\to \alpha) ],$ \\
\end{tabular}
\end{center}
and hence it is enough to find some $z'\in Z$ such that 
\[(R_\Diamond^{[0]}[\{\beta /  z \}])^{[0]}(z')\to (E(z', w)\to \alpha)\leq \beta\to (R_{\Diamond}(w, z) \to \alpha),\]
i.e.
\begin{center}
{\small
\begin{tabular}{cl l}
&$\left(\bigwedge_{(\gamma, v)\in Z_X}[\beta\to (R_\Diamond(v, z)\to\gamma)]\to (E(z', v)\to\gamma)\right)$\\
  & $\to (E(z', w)\to \alpha) \leq \beta\to (R_{\Diamond}(w, z) \to \alpha)$ & ($\ast$)\\
\end{tabular}
}
\end{center}
Let $z'\in Z$ such that $u_{z'}: {\mathbf{Fm}_{\ast}}\to\mathbf{A}$ maps $\bot$ to $0$ and every other $\phi\in {\mathbf{Fm}_{\ast}}$ to $1$, and $f_{z'}: = f_z^{-\Diamond}$ (cf.~Lemma \ref{lemma:f minus diam is filter}.2). Then
\begin{center}
\begin{tabular}{rcl}
$E(z', w)$ & = & $\bigwedge_{\phi\in \mathbf{Fm}} f_z^{-\Diamond}(\phi)\to u_{w}(\phi)$\\
& = & $R_{\Diamond}(w, z)$,\\
\end{tabular}
\end{center}
and likewise $E(z', v) = R_{\Diamond}(v, z)$. Therefore, for this choice of $z'$, inequality ($\ast$) can be rewritten as follows:
\begin{center}
\begin{tabular}{cl }
&$\left(\bigwedge_{(\gamma, v)\in Z_X}[\beta\to (R_\Diamond(v, z)\to\gamma)]\to (R_{\Diamond}(v, z)\to\gamma)\right)$\\
  & $\to (R_{\Diamond}(w, z)\to \alpha) \leq \beta\to (R_{\Diamond}(w, z) \to \alpha)$ \\
\end{tabular}
\end{center}
The inequality above is true if \[\beta \leq \bigwedge_{(\gamma, v)\in Z_X}[\beta\to (R_\Diamond(v, z)\to\gamma)]\to (R_{\Diamond}(v, z)\to\gamma),\]
i.e.~if for every $(\gamma, v)\in Z_X$, 
\[\beta \leq [\beta\to (R_\Diamond(v, z)\to\gamma)]\to (R_{\Diamond}(v, z)\to\gamma),\]
which is an instance of a tautology in residuated lattices.

Let us show that $(R_\Diamond^{[1]}[\{\beta / (\alpha, z) \}])^{[10]} \subseteq R_\Diamond^{[1]}[\{\beta /  (\alpha, z) \}]$. By definition, for every $w\in Z_A$,

\begin{center}
{\small
\begin{tabular}{cl}
& $R_\Diamond^{[1]}[\{\beta /  (\alpha, z)\}] (w)$\\
= &$\bigwedge_{(\gamma, z')\in Z_X}[\{\beta /  (\alpha, z)\}(\gamma, z')\to (R_{\Diamond}(z', w)\to \gamma)]$\\
= &$\beta\to (R_{\Diamond}(z, w)\to \alpha)$\\
&\\
& $(R_\Diamond^{[1]}[\{\beta / (\alpha, z) \}])^{[10]}(w)$\\
= &$\bigwedge_{(\gamma, z')\in Z_X}[(R_\Diamond^{[1]}[\{\beta / (\alpha, z) \}])^{[1]}(\gamma, z')\to (E(w, z')\to \gamma)]$.\\
\end{tabular}
}
\end{center}
Hence it is enough to find some $(\gamma, z')\in Z_X$ such that 
\begin{multline*}
(R_\Diamond^{[1]}[\{\beta / (\alpha, z) \}])^{[1]}(\gamma, z')\to (E(w, z')\to \gamma) \\
\leq \beta\to ( R_{\Diamond}(z, w)\to \alpha),
\end{multline*}
i.e.
\begin{center}
\begin{tabular}{cll}
& $\left( \bigwedge_{v\in Z} (\beta\to (R_{\Diamond}(z, v)\to \alpha))\to (E(v, z')\to \gamma)\right)$\\
&$\to (E(w, z')\to \gamma)\leq \beta\to (R_{\Diamond}(z, w)\to \alpha)$ & ($\ast$)\\
\end{tabular}
\end{center}

Let $(\gamma, z') := (\alpha, z')$ such that $f_{z'}: \mathbf{Fm}_{\ast}\to\mathbf{A}$ maps $\top$ to $1$ and every other $\phi\in \mathbf{Fm}_{\ast}$ to $0$, and $u_{z'}: \mathbf{Fm}_{\ast} \to\mathbf{A}$ is  defined by the assignment 
\[
u_{z'}(\phi) = \left\{\begin{array}{ll}
1 & \text{if }  \top\vdash \phi\\
u_z(\Diamond\phi) & \text{otherwise. } \\
\end{array}\right.
\]
by definition,  $u_{z'}$ maps $\top$ to 1 and $\bot$ to 0; moreover, using Lemma \ref{lemma:f minus diam is filter}.5, it can be readily verified that $u_{z'}$ is $\vee$-preserving.  Then 
\begin{center}
\begin{tabular}{rcl}
$E(v, z')$ & = & $\bigwedge_{\phi\in \mathbf{Fm}_{\ast}} (f_v(\phi)\to u_{z'}(\phi))$\\
& = & $\bigwedge_{\phi\in \mathbf{Fm}_{\ast}} (f_v(\phi)\to u_{z}(\Diamond \phi))$\\
& = & $\bigwedge_{\phi\in \mathbf{Fm}_{\ast}} (f_v^{-\Diamond}(\phi)\to u_{z}(\phi))$\\
& = & $R_{\Diamond}(z, v)$,\\
\end{tabular}
\end{center}
and likewise $E(w, z') = R_{\Diamond}(z, w)$. Therefore, for this choice of $z'$, inequality ($\ast$) can be rewritten as follows:
\begin{center}
\begin{tabular}{cll}
& $\left( \bigwedge_{v\in Z} (\beta\to R_{\Diamond}(z, v)\to \alpha)\to (R_{\Diamond}(z, v)\to \alpha)\right)$\\
&$\to (R_{\Diamond}(z, w)\to \alpha)\leq \beta\to (R_{\Diamond}(z, w)\to \alpha)$ & ($\ast$)\\
\end{tabular}
\end{center}
which is shown to be true by the same argument as the one concluding the verification of  the previous inclusion. 

Let us show that $(R_\Box^{[0]}[\{\beta /  (\alpha, z) \}])^{[10]} \subseteq R_\Box^{[0]}[\{\beta / (\alpha, z) \}] $.
 For any $w \in Z_A$,
\begin{center}
{\small
\begin{tabular}{c l}
&$R_{\Box}^{[0]}[\{\beta / (\alpha,z) \}](w)$ \\
 = & $\bigwedge_{(\gamma,z')\in Z_X}[\{\beta / (\alpha,z) \}(\gamma,z')\to  (R_{\Box}(w, z') \to \gamma)]$\\
  = & $\beta\to (R_{\Box}(w, z) \to \alpha)$,\\
  & \\
 & $(R_\Box^{[0]}[\{\beta /(\alpha,z) \}])^{[10]} (w)$\\
 = & $\bigwedge_{(\gamma,z')\in Z_X}[(R_\Box^{[0]}[\{\beta /(\alpha,z)\}])^{[1]}(\gamma,z')\to (E(w,z')\to \gamma) ].$ \\
\end{tabular}
}
\end{center}
Hence, it is enough to find some $(\gamma, z')\in Z_X$ such that 
\begin{center}
\begin{tabular}{c l}
& $(R_\Box^{[0]}[\{\beta /(\alpha,z)\}])^{[1]}(\gamma,z')\to (E(w,z')\to \gamma)$\\
$ \leq  $&$\beta\to (R_{\Box}(w, z) \to \alpha)$,\\
\end{tabular}
\end{center}
i.e.
\begin{center}
\begin{tabular}{cll}
& $\left( \bigwedge_{v\in Z} (\beta\to (R_{\Box}(v, z)\to \alpha))\to (E(v, z')\to \gamma)\right)$\\
&$\to (E(w,z')\to \gamma) \leq  \beta\to (R_{\Box}(w, z) \to \alpha)$ & ($\ast$)\\
\end{tabular}
\end{center}

Let $(\gamma, z') := (\alpha, z')$ such that $f_{z'}: \mathbf{Fm}_{\ast}\to\mathbf{A}$ maps $\top$ to $1$ and every other $\phi\in \mathbf{Fm}_{\ast}$ to $0$, and $u_{z'}: = u^{-\Box}_{z}$ (cf.~Lemma \ref{lemma:f minus diam is filter}.4). Then
\begin{center}
\begin{tabular}{rcl}
$E(v, z')$ & = & $\bigwedge_{\phi\in \mathbf{Fm}} [f_v(\phi)\to u_{z}^{-\Box}(\phi)]$\\
& = & $R_{\Box}(v, z)$,\\
\end{tabular}
\end{center}
and likewise $E(w, z') = R_{\Box}(w, z)$. Therefore, for this choice of $z'$, inequality ($\ast$) can be rewritten as follows:
 
\begin{center}
\begin{tabular}{cll}
& $\left( \bigwedge_{v\in Z} (\beta\to (R_{\Box}(v, z)\to \alpha))\to (R_{\Box}(v, z)\to \alpha)\right)$\\
&$\to (R_{\Box}(w, z) \to \alpha) \leq  \beta\to (R_{\Box}(w, z) \to \alpha)$. & \\
\end{tabular}
\end{center}
The inequality above is true if \[\beta \leq \bigwedge_{v\in Z} (\beta\to (R_{\Box}(v, z)\to \alpha))\to (R_{\Box}(v, z)\to \alpha),\]
i.e.~if for every $v\in Z_A$, 
\[\beta \leq (\beta\to (R_{\Box}(v, z)\to \alpha))\to (R_{\Box}(v, z)\to \alpha),\]
which is an instance of a tautology in residuated lattices.

For the last inclusion,   for any $(\alpha, w)\in Z_X$,
\begin{center}
\begin{tabular}{c l}
&$R_{\Box}^{[1]}[\{\beta / z \}](\alpha,w)$ \\
 = & $\bigwedge_{z'\in Z_A}[\{\beta / z \}(z')\to  (R_{\Box}(z',w) \to \alpha)]$\\
  = & $\beta\to (R_{\Box}(z, w) \to \alpha)$,\\
  & \\
 & $(R_\Box^{[1]}[\{\beta /  z \}])^{[01]} (\alpha, w)$\\
 = & $\bigwedge_{z'\in Z_A}[(R_\Box^{[1]}[\{\beta /  z \}])^{[0]}(z')\to (E(z', w)\to \alpha) ]$, \\
\end{tabular}
\end{center}
and hence it is enough to find some $z'\in Z_A$ such that 
\[(R_\Box^{[1]}[\{\beta /  z \}])^{[0]}(z')\to (E(z', w)\to \alpha) \leq \beta\to (R_{\Box}(z, w) \to \alpha), \]
i.e.
\begin{center}
{\small
\begin{tabular}{cl l}
&$\left(\bigwedge_{(\gamma, v)\in Z_X}[\beta\to (R_\Box(z, v)\to\gamma)]\to (E(z', v)\to\gamma)\right)$\\
  & $\to (E(z', w)\to \alpha) \leq \beta\to (R_{\Box}(z, w) \to \alpha)$ & ($\ast$)\\
\end{tabular}
}
\end{center}
Let $z'\in Z_A$ such that $u_{z'}: \mathbf{Fm}_{\ast}\to\mathbf{A}$ maps $\bot$ to $0$ and every other $\phi\in \mathbf{Fm}_{\ast}$ to $1$, and $f_{z'}: \mathbf{Fm}_{\ast}\to\mathbf{A}$ is  defined by the assignment 
\[
f_{z'}(\phi) = \left\{\begin{array}{ll}
0 & \text{if }   \phi\vdash \bot\\
f_z(\Box\phi) & \text{otherwise. } \\
\end{array}\right.
\]
by definition,  $f_{z'}$ maps $\top$ to 1 and $\bot$ to 0; moreover, using Lemma \ref{lemma:f minus diam is filter}.6, it can be readily verified that $f_{z'}$ is $\wedge$-preserving.  Then 
\begin{center}
\begin{tabular}{rcl}
$E(z', w)$ & = & $\bigwedge_{\phi\in \mathbf{Fm}_{\ast}} (f_{z'}(\phi)\to u_{w}(\phi))$\\
& = & $\bigwedge_{\phi\in \mathbf{Fm}_{\ast}} (f_z(\Box\phi)\to u_{w}(\phi))$\\
& = & $\bigwedge_{\phi\in \mathbf{Fm}_{\ast}} (f_z(\phi)\to u_{w}^{-\Box}(\phi))$\\
& = & $R_{\Box}(z, w)$,\\
\end{tabular}
\end{center}
and likewise $E(z', v) = R_{\Box}(z, v)$. Therefore, for this choice of $z'$, inequality ($\ast$) can be rewritten as follows:

\begin{center}
\begin{tabular}{cl l}
&$\left(\bigwedge_{(\gamma, v)\in Z_X}[\beta\to (R_\Box(z, v)\to\gamma)]\to (R_\Box(z, v)\to\gamma)\right)$\\
  & $\to (R_{\Box}(z, w) \to \alpha) \leq \beta\to (R_{\Box}(z, w) \to \alpha)$. & \\
\end{tabular}
\end{center}

The inequality above is true if \[\beta \leq \bigwedge_{(\gamma, v)\in Z_X}[\beta\to (R_\Box(z, v)\to\gamma)]\to (R_\Box(z, v)\to\gamma),\]
i.e.~if for every $(\gamma, v)\in Z_X$, 
\[\beta \leq [\beta\to (R_\Box(z, v)\to\gamma)]\to (R_\Box(z, v)\to\gamma),\]
which is an instance of a tautology in residuated lattices.
\end{proof}
\begin{definition}
\label{def:canonical model}
The {\em canonical graph-based} $\mathbf{A}$-{\em model} associated with $\mathbf{Fm}_{\ast}$ is the structure  $\mathbb{M}_{\mathbf{Fm}_{\ast}} =(\mathbb{G}_{\mathbf{Fm}_{\ast}}, V)$ such that $\mathbb{G}_{\mathbf{Fm}_{\ast}}$ is the canonical graph-based $\mathbf{A}$-frame of Definition \ref{def:canonical frame}, and  if $p \in \mathsf{Prop}$, then $V(p) = (\val{p}, \descr{p})$ with $\val{p}: Z_A\to \mathbf{A}$ and $\descr{p}: Z_X\to \mathbf{A}$ defined by  $z\mapsto f_z(p)$ and  $(\alpha, z)\mapsto u_z(p)\to \alpha$, respectively.\footnote{We write $\mathbb{M}$ for $\mathbb{M}_{\mathbf{Fm}_{\ast}}$ when $\mathbf{Fm}_{\ast}$ is clear from the context.}
\end{definition}
\begin{lemma}
	The structure $\mathbb{G}_{\mathbf{Fm}_{\ast}}$ of Definition \ref{def:canonical model}
 is a graph-based $\mathbf{A}$-model.
	\end{lemma}
	\begin{proof}
It is enough to show that $\val{p}^{[1]}=\descr{p}$ and $\val{p}=\descr{p}^{[0]}$  for any $p \in \mathsf{Prop}$.  To show that $\descr{p} (\alpha, z)\leq \val{p}^{[1]}(\alpha, z)$ for any $(\alpha, z)\in Z_X$, by definition, we need to show that
\[u_z(p)\to\alpha\leq \bigwedge_{z'\in Z_A}(\val{p}(z')\to (E(z', z)\to\alpha)),\]
i.e.~that for every $z'\in Z_A$, 
\[u_z(p)\to\alpha\leq \val{p}(z')\to (E(z', z)\to\alpha).\]
By definition, the inequality above is equivalent to
\[u_z(p)\to\alpha\leq f_{z'}(p)\to (\bigwedge_{\phi\in \mathbf{Fm}_{\ast}} [f_{z'}(\phi)\to u_z(\phi)]\to\alpha).\]
Since $\bigwedge_{\phi\in \mathbf{Fm}_{\ast}} [f_{z'}(\phi)\to u_z(\phi)]\leq f_{z'}(p)\to u_z(p)$ and $\to$ is order-reversing in its first coordinate, it is enough to show that 
\[u_z(p)\to\alpha\leq f_{z'}(p)\to [(f_{z'}(p)\to u_z(p))\to\alpha].\]
By residuation the inequality above is equivalent to
\[u_z(p)\to\alpha\leq [f_{z'}(p)\otimes (f_{z'}(p)\to u_z(p))]\to\alpha,\]
which is equivalent to
\[[f_{z'}(p)\otimes (f_{z'}(p)\to u_z(p))]\otimes [u_z(p)\to\alpha]\leq \alpha,\]
which is the instance of a tautology in residuated lattices.
Conversely, to show that $ \val{p}^{[1]}(\alpha, z)\leq \descr{p} (\alpha, z)$, i.e.
\[\bigwedge_{z'\in Z_A}(\val{p}(z')\to (E(z', z)\to\alpha))\leq u_z(p)\to\alpha,\]
it is enough to show that
\begin{equation}
\label{eqq}
\val{p}(z')\to (E(z', z)\to\alpha))\leq u_z(p)\to\alpha
\end{equation}
for some $z'\in Z$. Let $z': = (f_p, u)$ such that $u:\mathbf{Fm}_{\ast}\to \mathbf{A}$  maps $\bot$ to $0$ and every other element of $\mathbf{Fm}_{\ast}$ to $1$, and  $f_p:\mathbf{Fm}_{\ast}\to \mathbf{A}$ is defined by the assignment 
\[
f_p(\phi) = \left\{\begin{array}{ll}
1 & \text{if }  p\vdash \phi\\
0 & \text{otherwise.} \\
\end{array}\right.
\]
 Hence, $E(z', z) = \bigwedge_{\phi\in \mathbf{Fm}_{\ast}}(f_p(\phi)\to u_z(\phi))=\bigwedge_{p\vdash\phi}u_z(\phi) = u_z(p)$, the last identity holding since $u_z$ is order-preserving.  
 Therefore,  
 $\val{p}(z')\to (E(z', z)\to\alpha)) = f_p(p)\to (u_z(p)\to \alpha) = 1\to (u_z(p)\to \alpha) = u_z(p)\to \alpha$, which shows  \eqref{eqq}.
 
By adjunction, the inequality $\descr{p} \leq \val{p}^{[1]}$ proven above implies that   $\val{p}\leq\descr{p}^{[0]}$. Hence, to show that  $\val{p}=\descr{p}^{[0]}$, it is enough to show $\descr{p}^{[0]}(z)\leq\val{p}(z)$ for every $z\in Z$, i.e. 
\[\bigwedge_{(\alpha, z')\in Z_X} \descr{p}(\alpha, z')\to (E(z, z')\to \alpha)\leq f_z(p),\]
and to show the inequality above, it is enough to show that 
\begin{equation}\label{eqqq} \descr{p}(\alpha, z')\to (E(z, z')\to \alpha)\leq f_z(p)\end{equation}
for some $(\alpha, z')\in Z_X$.
Let $\alpha: =f_z(p)$ and $z':  = (f_{z'}, u_{p})$ be such that 
$u_{z'} = u_{p}: \mathbf{Fm}_{\ast}\to \mathbf{A}$ is defined by the following assignment:
 \[
u_{p}(\phi) = \left\{\begin{array}{ll}
0 & \text{if } \phi\vdash \bot\\
f_z(p) & \text{if } \phi\vdash p \mbox{ and } \phi \not\vdash \bot\\
1 & \text{if } \phi\not\vdash p.
\end{array}\right.
\]
By construction, $u_{z'}$ is $\vee$-, $\bot$- and $\top$-preserving. Moreover, $\descr{p}(\alpha, z') = u_{z'}(p)\to\alpha = f_z(p)\to f_z(p) = 1$, and $E(z, z') = \bigwedge_{\phi\in \mathbf{Fm}_{\ast}}(f_z(\phi)\to u_{z'}(\phi))  = \bigwedge_{\phi\vdash p}(f_z(\phi)\to f_z(p))  = 1$. 
 Hence, the left-hand side of \eqref{eqqq} can be equivalently rewritten as $1\to (1\to f_z(p)) = f_z(p)$, which shows \eqref{eqqq} and concludes the proof.
\end{proof}

Recall that $\mathbf{Fm}_0$ is the Lindenbaum-Tarski algebra of the logic $\mathbf{L}_0$ which is the axiomatic extension of  $\mathbf{L}$ with the axiom $\Box \bot \vdash \bot$. We write $\mathbf{L}_{0\ast}$ to denote an arbitrary (possible non-proper) extension of $\mathbf{L}_{0}$ and $\mathbf{Fm}_{0\ast}$ for the corresponding Lindenbaum-Tarski algebra. The axiom $\Box \bot \vdash \bot$ is required only in the inductive step for $\Box$-formulas in the following lemma. To emphasise this, we work mostly in the context of $\mathbf{Fm}_{\ast}$ and switch to $\mathbf{Fm}_{0\ast}$ only when required. 

\begin{lemma}[Truth Lemma]
\label{lemma:truthlemma}
	For every  $\phi\in \mathbf{Fm}_{0\ast}$, the maps $\val{\phi}: Z_A\to \mathbf{A}$ and $\descr{\phi}: Z_X\to \mathbf{A}$ coincide with those defined by the assignments $z\mapsto f_z(\phi)$ and $(\alpha, z)\mapsto u_z(\phi)\to \alpha$, respectively. 
\end{lemma}
\begin{proof}
We proceed by  induction on $\phi$. If $\phi: = p\in \mathsf{Prop}$, the statement follows immediately from  Definition \ref{def:canonical model}. 

If $\phi: =\top$, then $\val{\top}(z) = 1 = f_z(\top)$ since $\mathbf{A}$-filters are $\top$-preserving. Moreover,
\begin{center}
\begin{tabular}{r cl}
 $\descr{\top}(\alpha, z)$& $=$& $\val{\top}^{[1]}(\alpha, z)$\\
 & $=$& $ \bigwedge_{z'\in Z_A} [\val{\top}(z')\to (E(z', z)\to\alpha)]$\\
  & $=$& $ \bigwedge_{z'\in Z_A} [f_{z'}(\top)\to (E(z', z)\to\alpha)]$\\
  & $=$& $ \bigwedge_{z'\in Z_A} [E(z', z)\to\alpha]$. \\
  \end{tabular}
 \end{center}
 So, to show that $u_z(\top)\to \alpha\leq \descr{\top}(\alpha, z)$, we need to show that for every $z'\in Z$,
\[u_z(\top)\to \alpha\leq E(z', z)\to\alpha,\]
and for this, it is enough to show that 
\[\bigwedge_{\psi\in \mathbf{Fm}_{\ast}}[f_{z'}(\psi)\to u_z(\psi)]\leq u_z(\top),\]
which is true, since by definition, $u_z(\top) = 1$.
To show that $\descr{\top}(\alpha, z)\leq  u_z(\top)\to \alpha$, i.e.~that \[\bigwedge_{z'\in Z_A} [E(z', z)\to\alpha]\leq u_z(\top)\to \alpha,\]
it is enough  to find some $z'\in Z$ such that  $E(z', z)\to\alpha\leq u_z(\top)\to \alpha$. Let $z': = (f_\top, u)$ such that $u:\mathbf{Fm}_{\ast}\to \mathbf{A}$ maps $\top$ to $1$ and every other element of $\mathbf{Fm}_{\ast}$ to $0$, and  $f_\top:\mathbf{Fm}_{\ast}\to \mathbf{A}$ is defined by the assignment 
\[
f_\top(\phi) = \left\{\begin{array}{ll}
1 & \text{if }  \top\vdash \phi\\
0 & \text{otherwise. } \\
\end{array}\right.
\]\
By definition, $E(z', z) = \bigwedge_{\psi\in \mathbf{Fm}_{\ast}}[f_{z'}(\psi)\to u_z(\psi)] =  \bigwedge_{\top\vdash \psi}[1\to u_z(\psi)]=  \bigwedge_{\top\vdash \psi}u_z(\psi)
\geq u_z(\top)$, the last inequality being due to the fact that  $u_z$ is order-preserving. 
Hence, $E(z', z)\to\alpha\leq u_z(\top)\to \alpha$, as required.

If $\phi: =\bot$, then $\descr{\bot}(\alpha, z) = 1 = u_z(\bot)\to \alpha$ since complements of $\mathbf{A}$-ideals are $\bot$-preserving. Let us show that $\val{\bot}(z) = f_z(\bot)$. The inequality $f_z(\bot)\leq \val{\bot}(z) $ follows immediately from the fact that $f_z$ is proper and hence $f_z(\bot) = 0$. To show that $\val{\bot}(z) \leq f_z(\bot)$, by definition  $\val{\bot}(z)=\descr{\bot}^{[0]}(z) =  \bigwedge_{(\alpha, z')\in Z_X} [(u_{z'}(\bot)\to\alpha)\to (E(z, z')\to\alpha)]$,
%
hence, it is enough to find some $(\alpha, z')\in Z_X$ such that  \begin{equation}\label{eqqqqqq}(u_{z'}(\bot)\to\alpha)\to (E(z, z')\to\alpha) \leq f_z(\bot).\end{equation}
Let $\alpha: = f_z(\bot)$ and  let $z': = (f_{\top}, u_{\bot})$ such that $f_{\top}:\mathbf{Fm}_{\ast}\to \mathbf{A}$ is defined as indicated above in the base case for $\phi: = \top$, and  $u_{\bot}:\mathbf{Fm}_{\ast} \to \mathbf{A}$ is defined by the assignment 
	\[
u_{\bot}(\psi) = \left\{\begin{array}{ll}
0 & \text{if } \psi\vdash \bot\\
1 & \text{if }\psi\not\vdash \bot.
\end{array}\right.
\]
By definition and since $f_z$ is order-preserving and $\bot$-preserving, $E(z, z') = \bigwedge_{\psi\in \mathbf{Fm}_{\ast}}[f_z(\psi)\to u_{\bot}(\psi)]  = 1$. Hence, \eqref{eqqqqqq} can be rewritten as follows:
\[(f_z(\bot)\to f_z(\bot))\to f_z(\bot)\leq f_z(\bot),\]
which is true since $f_z(\bot)\to f_z(\bot) =1$ and $1\to f_z(\bot) = f_z(\bot)$.

If  $\phi: = \phi_1\wedge\phi_2$, then $\val{\phi_1\wedge\phi_2}(z) = (\val{\phi_1} \wedge \val{\phi_2})(z) = \val{\phi_1}(z) \wedge \val{\phi_2}(z) = f_z(\phi_1) \wedge f_z(\phi_2) = f_z(\phi_1 \wedge \phi_2)$. 
Let us  show that $\descr{\phi_1 \wedge \phi_2}(\alpha, z) = u_z(\phi_1 \wedge \phi_2)\to \alpha$. By definition,
\begin{center}
\begin{tabular}{cl}
& $\descr{\phi_1 \wedge \phi_2}(\alpha, z)$\\
= & $\val{\phi_1 \wedge \phi_2}^{[1]}(\alpha, z)$\\
= & $\bigwedge_{z'\in Z}[\val{\phi_1 \wedge \phi_2}(z')\to (E(z', z)\to \alpha)]$\\
= & $\bigwedge_{z'\in Z}[f_{z'}(\phi_1 \wedge \phi_2)\to (E(z', z)\to \alpha)]$.\\
\end{tabular}
\end{center}
Hence,  to show that $u_z(\phi_1 \wedge \phi_2)\to \alpha\leq \descr{\phi_1 \wedge \phi_2}(\alpha, z)$, we need to show that for every $z'\in Z$,
\[u_z(\phi_1 \wedge \phi_2)\to \alpha\leq f_{z'}(\phi_1\wedge \phi_2)\to (E(z', z)\to \alpha).\]
Since by definition $E(z', z) = \bigwedge_{\phi\in \mathbf{Fm}_{\ast}}[f_{z'}(\phi)\to u_z(\phi)]\leq f_{z'}(\phi_1 \wedge \phi_2)\to u_z(\phi_1 \wedge \phi_2)$ and $\to$ is order-reversing in the first coordinate and order-preserving in the second one, 
it is enough to show that for every $z'\in Z$,
\begin{center}
\begin{tabular}{cl}
& $u_z(\phi_1 \wedge \phi_2)\to \alpha$\\
$\leq$ &$ f_{z'}(\phi_1\wedge \phi_2)\to ((f_{z'}(\phi_1 \wedge \phi_2)\to u_z(\phi_1 \wedge \phi_2))\to \alpha)$.\\
\end{tabular}
\end{center}

By residuation, the above inequality is equivalent to
\begin{center}
\begin{tabular}{cl}
& $u_z(\phi_1 \wedge \phi_2)\to \alpha$\\
$\leq$ &$[f_{z'}(\phi_1\wedge \phi_2)\otimes (f_{z'}(\phi_1 \wedge \phi_2)\to u_z(\phi_1 \wedge \phi_2))]\to \alpha$.\\
\end{tabular}
\end{center}
The above inequality is true if 
\[f_{z'}(\phi_1\wedge \phi_2)\otimes (f_{z'}(\phi_1 \wedge \phi_2)\to u_z(\phi_1 \wedge \phi_2))\leq u_z(\phi_1 \wedge \phi_2),\]
which is an instance of a tautology in residuated lattices.

To show that $\descr{\phi_1 \wedge \phi_2}(\alpha, z)\leq u_z(\phi_1 \wedge \phi_2)\to \alpha$, it is enough to find some $z'\in Z$ such that 
\[f_{z'}(\phi_1 \wedge \phi_2)\to (E(z', z)\to \alpha)\leq u_z(\phi_1 \wedge \phi_2)\to \alpha.\]
Let $z': = (f_{\phi_1 \wedge \phi_2}, u_\bot)$ such that $u_\bot:\mathbf{Fm}_{\ast}\to \mathbf{A}$ 
is defined as indicated above in the base case for $\phi: = \bot$,
and  $f_{\phi_1 \wedge \phi_2}:\mathbf{Fm}_{\ast}\to \mathbf{A}$ is defined by the assignment 
\[
f_{\phi_1 \wedge \phi_2}(\psi) = \left\{\begin{array}{ll}
1 & \text{if }  \phi_1 \wedge \phi_2\vdash \psi\\
0 & \text{otherwise. } \\
\end{array}\right.
\]

For $z': = z$, since $f_{z'}(\phi_1 \wedge \phi_2) = 1$ and $1\to (E(z', z)\to \alpha) = E(z', z)\to \alpha$,  the inequality above becomes
\[E(z', z)\to \alpha\leq u_z(\phi_1 \wedge \phi_2)\to \alpha,\]
to verify which, it is enough to show that $u_z(\phi_1 \wedge \phi_2)\leq E(z', z)$. Indeed, 
by definition, $E(z', z) =  \bigwedge_{\psi\in \mathbf{Fm}_{\ast}}[f_{z'}(\psi)\to u_z(\psi)] =  \bigwedge_{\phi_1 \wedge \phi_2\vdash \psi}[1\to u_z(\psi)]=  \bigwedge_{\phi_1 \wedge \phi_2\vdash \psi}u_z(\psi)\geq u_z(\phi_1 \wedge \phi_2)$, the last inequality being due to the fact that  $u_z$ is order-preserving.

If  $\phi: = \phi_1\vee\phi_2$, then $\descr{\phi_1\vee\phi_2}(\alpha, z) = (\descr{\phi_1} \wedge \descr{\phi_2})(\alpha, z) = \descr{\phi_1}(\alpha, z) \wedge \descr{\phi_2}(\alpha, z) = (u_z(\phi_1)\to\alpha) \wedge (u_z(\phi_2)\to \alpha) = (u_z(\phi_1)\vee u_z(\phi_2))\to \alpha)=  u_z(\phi_1 \vee \phi_2)\to \alpha$. 
Let us  show that $\val{\phi_1 \vee \phi_2}(z) = f_z(\phi_1 \vee \phi_2)$. By definition,
\begin{center}
\begin{tabular}{cl}
& $\val{\phi_1 \vee \phi_2}(z)$\\
= & $\descr{\phi_1 \vee \phi_2}^{[0]}(z)$\\
= & $\bigwedge_{(\alpha, z')\in Z_X}[\descr{\phi_1 \vee \phi_2}(\alpha, z')\to (E(z, z')\to \alpha)]$\\
= & $\bigwedge_{(\alpha, z')\in Z_X}[(u_{z'}(\phi_1 \vee \phi_2)\to\alpha)\to (E(z, z')\to \alpha)]$.\\
\end{tabular}
\end{center}
Hence,  to show that $f_z(\phi_1 \vee \phi_2)\leq \val{\phi_1 \vee \phi_2}(z)$, we need to show that for every $(\alpha, z')\in Z_X$,
\[f_z(\phi_1 \vee \phi_2)\leq (u_{z'}(\phi_1 \vee \phi_2)\to\alpha)\to (E(z, z')\to \alpha).\]
Since by definition $E(z, z') = \bigwedge_{\psi\in \mathbf{Fm}_{\ast}}[f_{z}(\psi)\to u_{z'}(\psi)]\leq f_{z}(\phi_1 \vee \phi_2)\to u_{z'}(\phi_1 \vee \phi_2)$ and $\to$ is order-reversing in the first coordinate and order-preserving in the second one, 
it is enough to show that for every $(\alpha, z')\in Z_X$,
\begin{multline*}
f_z(\phi_1 \vee \phi_2) \leq \\ 
(u_{z'}(\phi_1 \vee \phi_2)\to\alpha)\to ((f_{z}(\phi_1 \vee \phi_2)\to u_{z'}(\phi_1 \vee \phi_2))\to \alpha).
\end{multline*}

By residuation, associativity and commutativity of $\otimes$, the inequality above is equivalent to
\begin{center}
\begin{tabular}{cl}
& $ f_z(\phi_1 \vee \phi_2)\otimes (f_{z}(\phi_1 \vee \phi_2)\to u_{z'}(\phi_1 \vee \phi_2))$\\
 &$ \otimes (u_{z'}(\phi_1 \vee \phi_2)\to\alpha) \leq \alpha$,\\
\end{tabular}
\end{center}
which is a tautology in residuated lattices.

To show that $\val{\phi_1 \vee \phi_2}(z)\leq f_z(\phi_1 \vee \phi_2)$, it is enough to find some $(\alpha, z')\in Z_X$ such that 
\begin{equation}
\label{eqqqqq}
(u_{z'}(\phi_1 \vee \phi_2)\to\alpha)\to (E(z', z)\to \alpha)\leq f_z(\phi_1 \vee \phi_2).\end{equation}
Let $\alpha: = f_z(\phi_1 \vee \phi_2)$ and  let $z': = (f_{\top}, u_{\phi_1 \vee \phi_2})$ such that $f_{\top}:\mathbf{Fm}_{\ast}\to \mathbf{A}$ is defined as indicated above in the base case for $\phi:= \top$,
and  $u_{\phi_1 \vee \phi_2}:\mathbf{Fm}_{\ast}\to \mathbf{A}$ is defined by the assignment 
	\[
u_{\phi_1 \vee \phi_2}(\psi) = \left\{\begin{array}{ll}
0 & \text{if } \psi\vdash \bot\\
f_z(\phi_1 \vee \phi_2) & \text{if }  \psi\not\vdash \bot \mbox{ and } \psi\vdash \phi_1 \vee \phi_2 \\
1 & \text{if }\psi\not\vdash \phi_1 \vee \phi_2.
\end{array}\right.
\]
By definition and since $f_z$ is order-preserving and proper, $E(z, z') = \bigwedge_{\psi\in \mathbf{Fm}_{\ast}}[f_z(\psi)\to u_{\phi_1 \vee \phi_2}(\psi)] = \bigwedge_{\bot\not\dashv \psi\vdash\phi_1 \vee \phi_2} [f_z(\psi)\to f_z(\phi_1 \vee \phi_2) ] = 1$. Hence, \eqref{eqqqqq} can be rewritten as follows:
\[(f_z(\phi_1 \vee \phi_2)\to f_z(\phi_1 \vee \phi_2))\to f_z(\phi_1 \vee \phi_2)\leq f_z(\phi_1 \vee \phi_2),\]
which is true since $f_z(\phi_1 \vee \phi_2)\to f_z(\phi_1 \vee \phi_2) =1$ and $1\to f_z(\phi_1 \vee \phi_2) = f_z(\phi_1 \vee \phi_2)$.

If $\phi: = \Diamond\psi$, let us show that $\descr{\Diamond\psi} (\alpha, z) = u_z(\Diamond \psi)\to \alpha$. By definition,
\begin{center}
\begin{tabular}{rcl}
$\descr{\Diamond\psi} (\alpha, z)$ &
 = & $R^{[0]}_\Diamond[\val{\psi}](\alpha, z)$\\
&= & $\bigwedge_{z'\in Z_A}[\val{\psi}(z')\to  (R_{\Diamond}(z, z') \to \alpha)]$\\
&= & $\bigwedge_{z'\in Z_A}[f_{z'}(\psi)\to  (R_{\Diamond}(z, z') \to \alpha)]$,\\
\end{tabular}
\end{center}
Hence,  to show that $u_z(\Diamond\psi)\to \alpha\leq \descr{\Diamond\psi}(\alpha, z)$, we need to show that for every $z'\in Z$,
\[u_z(\Diamond\psi)\to \alpha\leq f_{z'}(\psi)\to  (R_{\Diamond}(z, z') \to \alpha).\]
By definition we have  $R_\Diamond(z, z') = \bigwedge_{\phi\in \mathbf{Fm}_{\ast}}(f_{z'}(\phi) \to u_{z}(\Diamond\phi)) 
\leq f_{z'}(\psi)\to u_{z}(\Diamond \psi)$, 
and since $\to$ is order-reversing in the first coordinate and order-preserving in the second one, 
it is enough to show that for every $z'\in Z$,
\[u_z(\Diamond\psi)\to \alpha\leq f_{z'}(\psi)\to  ((f_{z'}(\psi)\to u_{z}(\Diamond \psi)) \to \alpha).\]
By residuation, associativity and commutativity of $\otimes$, the inequality above is equivalent to
\[ [f_{z'}(\psi)\otimes (f_{z'}(\psi)\to u_{z}(\Diamond \psi))] \otimes (u_z(\Diamond\psi)\to \alpha) \leq \alpha\]
which is a tautology in residuated lattices.

To show that $\descr{\Diamond\psi}(\alpha, z)\leq u_z(\Diamond\psi)\to \alpha$, it is enough to find some $z'\in Z$ such that
\begin{equation}
\label{eqqqqqqq}
f_{z'}(\psi)\to  (R_{\Diamond}(z, z') \to \alpha)\leq u_z(\Diamond\psi)\to \alpha.\end{equation}
Let $z': = (f_{\psi}, u_\bot)$ such that $u_\bot:\mathbf{Fm}_{\ast}\to \mathbf{A}$ 
is defined as indicated above in the base case for $\phi: = \bot$,
 and  $f_{\psi}:\mathbf{Fm}_{\ast}\to \mathbf{A}$ is defined by the assignment 
\[
f_{\psi}(\phi) = \left\{\begin{array}{ll}
1 & \text{if }  \psi\vdash \phi\\
0 & \text{otherwise. } \\
\end{array}\right.
\]
By definition and Lemma \ref{eq:premagicnew2},
\begin{center}
\begin{tabular}{rcl}
$R_\Diamond(z, z')$ &  = &$\bigwedge_{\phi\in \mathbf{Fm}_{\ast}}(f_{z'}^{-\Diamond}(\phi)\to  u_{z}(\phi)) $\\
&=&$ \bigwedge_{\phi\in \mathbf{Fm}_{\ast}}(f_{z'}(\phi)\to  u_{z}(\Diamond \phi))$\\
&=&$ \bigwedge_{\psi\vdash \phi} u_{z}(\Diamond \phi)$\\
&$\geq$&$  u_{z}(\Diamond \psi)$,\\
\end{tabular}
\end{center}
the last inequality being due to the fact that $u_z$ and $\Diamond$ are order-preserving. Since $\to$ is order reversing in the first coordinate and order-preserving in the second one, to show \eqref{eqqqqqqq} it is enough to show that 
\[f_{z'}(\psi)\to  (u_{z}(\Diamond \psi) \to \alpha)\leq u_z(\Diamond\psi)\to \alpha.\]
This immediately follows from the fact that, by construction,  $f_{z'}(\psi) = 1$.

Let us show that $\val{\Diamond\psi}(z) = f_z(\Diamond\psi)$. By definition,
\begin{center}
\begin{tabular}{cl}
&$\val{\Diamond\psi} (z)$\\
  = & $\descr{\Diamond\psi}^{[0]}(z)$\\ 
 = & $\bigwedge_{(\alpha, z')\in Z_X}[\descr{\Diamond\psi}(\alpha, z')\to (E(z, z')\to\alpha)]$\\
 = & $\bigwedge_{(\alpha, z')\in Z_X}[(u_{z'}(\Diamond\psi)\to \alpha)\to (E(z, z')\to\alpha)].$\\
\end{tabular}
\end{center}
Hence,  to show that $f_z(\Diamond\psi)\leq \val{\Diamond\psi}(z)$, we need to show that for every $(\alpha, z')\in Z_X$,
\[f_z(\Diamond\psi)\leq (u_{z'}(\Diamond\psi)\to \alpha)\to (E(z, z')\to\alpha).\]
Since by definition $E(z, z') = \bigwedge_{\phi\in \mathbf{Fm}_{\ast}}[f_{z}(\phi)\to u_{z'}(\phi)]\leq f_{z}(\Diamond\psi)\to u_{z'}(\Diamond\psi)$ and $\to$ is order-reversing in the first coordinate and order-preserving in the second one, 
it is enough to show that for every $(\alpha, z')\in Z_X$,
\[f_z(\Diamond\psi)\leq (u_{z'}(\Diamond\psi)\to \alpha)\to ((f_{z}(\Diamond\psi)\to u_{z'}(\Diamond\psi))\to\alpha).\]
By residuation, associativity and commutativity of $\otimes$, the inequality above is equivalent to
\[ [f_z(\Diamond\psi)\otimes (f_{z}(\Diamond\psi)\to u_{z'}(\Diamond\psi))] \otimes (u_{z'}(\Diamond\psi)\to \alpha) \leq \alpha\]
which is a tautology in residuated lattices.

To show that $\val{\Diamond\psi}(z)\leq f_z(\Diamond\psi)$, it is enough to find  some $(\alpha, z')\in Z_X$ such that 
\begin{equation}
\label{eqqqqqqqq}(u_{z'}(\Diamond\psi)\to \alpha)\to (E(z, z')\to\alpha)\leq f_z(\Diamond\psi).\end{equation}
Let $\alpha: = f_z(\Diamond\psi)$ and 
  let $z': = (f_{\top}, u_{\Diamond\psi})$ such that $f_{\top}:\mathbf{Fm}_{\ast}\to \mathbf{A}$ is defined as indicated above in the base case for $\phi: = \top$, and  $u_{\Diamond\psi}:\mathbf{Fm}_{\ast}\to \mathbf{A}$ is defined by the assignment 
	\[
u_{\Diamond\psi}(\phi) = \left\{\begin{array}{ll}
0 & \text{if } \phi\vdash \bot\\
f_z(\Diamond\psi) & \text{if }  \phi\not\vdash \bot \mbox{ and } \phi\vdash \Diamond\psi \\
1 & \text{if }\phi\not\vdash \Diamond\psi.
\end{array}\right.
\]
By definition and since $f_z$ is order-preserving and proper, $E(z, z') = \bigwedge_{\phi\in \mathbf{Fm}_{\ast}}[f_z(\phi)\to u_{\Diamond\psi}(\phi)] = \bigwedge_{\phi\vdash\Diamond\psi} [f_z(\phi)\to f_z(\Diamond\psi) ] = 1$. Hence, \eqref{eqqqqqqqq} can be rewritten as follows:
\[(f_z(\Diamond\psi)\to f_z(\Diamond\psi))\to f_z(\Diamond\psi)\leq f_z(\Diamond\psi),\]
which is true since $f_z(\Diamond\psi)\to f_z(\Diamond\psi) =1$ and $1\to f_z(\Diamond\psi) = f_z(\Diamond\psi)$.

If $\phi:=\Box \psi$ we will show that $\val{\Box \psi}(z)=f_z(\Box \psi)$. By definition:
\begin{center}
{\small
\begin{tabular}{rcl}
$\val{\Box\psi} (z)$ & = & $R^{[0]}_\Box[\descr{\psi}](z)$\\
& = & $\bigwedge_{(\alpha,z') \in Z_X} [\descr{\psi}(\alpha,z') \to (R_{\Box}(z,z')\to \alpha)]$\\
& = & $\bigwedge_{(\alpha,z') \in Z_X} [(u_{z'}(\psi) \to \alpha) \to (R_{\Box}(z,z')\to \alpha)]$
\end{tabular}
}
\end{center}
Hence to show that $f_z(\Box \psi) \leq \val{\Box \psi}(z)$ we must show that for every $(\alpha,z') \in Z_X$ we have
\[f_z(\Box \psi) \leq (u_{z'}(\psi) \to \alpha) \to (R_{\Box}(z,z')\to \alpha).\] 

We have 
$R_{\Box}(z,z')=\bigwedge_{\phi \in \mathbf{Fm}_{0\ast}} (f_z(\Box \phi) \to u_{z'}(\phi))$ so 
$R_{\Box}(z,z')\leq f_z(\Box \psi) \to u_{z'}(\psi)$. 
Since $\to$ is order-reversing in the first coordinate and order-preserving in the 
second coordinate, it will be enough to show that 
\[ f_z(\Box \psi) \leq (u_{z'}(\psi) \to \alpha) \to ((f_z(\Box \psi) \to u_{z'}(\psi))\to \alpha).\]
Using residuation, and associativity and commutativty of $\otimes$ we can see that this is an instance of a tautology in residuated lattices. 

To show $\val{\Box \psi}(z)\leq f_z(\Box \psi)$ we must find $(\alpha, z') \in Z_X$ such that 

\begin{equation}\label{eq:Box1}
 (u_{z'}(\psi) \to \alpha) \to (R_{\Box}(z,z')\to \alpha) \leq f_z(\Box \psi).
\end{equation}

Let $\alpha: = f_z(\Box \psi)$ and   let $z': = (f_{\top}, u_{\psi})$ such that $f_{\top}:\mathbf{Fm}_{0\ast}\to \mathbf{A}$ is defined as indicated above in the base case for $\phi:= \top$,
and  $u_{\psi}:\mathbf{Fm}_{0\ast}\to \mathbf{A}$ is defined by the assignment 
	\[
u_{\psi}(\chi) = \left\{\begin{array}{ll}
0 & \text{if } \chi\vdash \bot\\
f_z(\Box\psi) & \text{if }  \chi\not\vdash \bot \mbox{ and } \chi\vdash  \psi \\
1 & \text{if }\chi\not\vdash \psi.
\end{array}\right.
\]
By definition, and since $f_{z}$ is order-preserving and $\Box\bot\leq \bot$ is valid,\footnote{This is the only place in the completeness proof where we are using an extra assumption.}  $R_{\Box}(z, z') = \bigwedge_{\chi\in\mathbf{Fm}_{0\ast}}(f_{z}(\Box\chi)\to u_{z'}(\chi)) = 1$. 
Hence, \eqref{eq:Box1} can be rewritten as follows:
\[(f_{z}(\Box\psi)\to f_z(\Box\psi)) \to (1\to f_z(\Box\psi)) \leq f_z(\Box \psi),\]
which is a tautology.\
Next, we want to show that $\descr{\Box \psi}(\alpha, z)=u_z(\Box\psi) \to \alpha$. By definition
\begin{center}
\begin{tabular}{rcl}
$\descr{\Box\psi} (\alpha,z)$ & = & $\val{\Box \psi}^{[1]}(\alpha,z)$\\
& = & $\bigwedge_{z' \in Z_A} [\val{\Box \psi}(z') \to (E(z',z) \to \alpha)]$\\
& = & $\bigwedge_{z' \in Z_A} [f_{z'}(\Box \psi) \to (E(z',z) \to \alpha)]$
\end{tabular}
\end{center}
To show that $\descr{\Box \psi}(\alpha, z) \leq u_z(\Box\psi) \to \alpha$ we just need to find $z' \in Z_A$ such 
that $f_{z'}(\Box \psi) \to (E(z',z) \to \alpha) \leq u_z(\Box \psi) \to \alpha$. Define $f_{\Box \psi} \colon \mathbf{Fm}_{0\ast} \to \mathbf{A}$
by
\[ f_{\Box \psi}(\chi) = \begin{cases} 1 & \text{ if } \Box \psi \vdash \chi \\ 0 & \text{otherwise.}\end{cases} \]
Now let $z'=(f_{\Box \psi}, u_{\bot})$ where $u_{\bot}$ is as defined in the base case. 
Clearly $f_{z'}(\Box \psi)=1$ and so 
$f_{z'}(\Box \psi) \to (E(z',z) \to \alpha) = E(z',z) \to \alpha$. Now 
\begin{center}
\begin{tabular}{rcl}
$E(z',z)$ & = & $\bigwedge_{\phi \in \mathbf{Fm}_{0\ast}} (f_{z'}(\phi) \to u_z(\phi))$ \\ 
& = & $\bigwedge_{\Box \psi \vdash \chi} (f_{z'}(\chi) \to u_z(\chi))$ \\
& = & $\bigwedge_{\Box \psi \vdash \chi} (1\to u_z(\chi)) $ \\
& = & $\bigwedge_{\Box \psi \vdash \chi} u_z(\chi) $ \\
& $\geq$ & $u_z(\Box \psi)$.
\end{tabular}
\end{center}
The last inequality follows from the fact that $u_z$ is order-preserving. Since $\to$ is order-reversing in the first coordinate we have
\[f_{z'}(\Box \psi) \to (E(z',z) \to \alpha) = E(z',z) \to \alpha \leq u_z(\Box \psi) \to \alpha.\]
 
To show that $ u_z(\Box\psi) \to \alpha\leq \descr{\Box \psi}(\alpha, z)$, we must show that 
for all $z' \in Z_A$ we have $u_z(\Box\psi) \to \alpha\leq f_{z'}(\Box \psi) \to (E(z',z) \to \alpha)$. By definition 
we have $E(z',z)=\bigwedge_{\phi \in \mathbf{Fm}_{0\ast}} (f_{z'}(\phi) \to u_{z}(\phi)) \leq f_{z'}(\Box \psi)\to u_{z}(\Box \psi)$. 
Therefore the desired inequality will follow if we can show 
\[
u_z(\Box\psi) \to \alpha\leq f_{z'}(\Box \psi) \to ((f_{z'}(\Box \psi)\to u_{z}(\Box \psi)) \to \alpha).
\]
By residuation, the above inequality is equivalent to 
\[ 
u_z(\Box\psi) \to \alpha\leq [f_{z'}(\Box \psi) \otimes (f_{z'}(\Box \psi)\to u_{z}(\Box \psi))] \to \alpha.
\]
This last inequality is true if 
\[ 
f_{z'}(\Box \psi) \otimes (f_{z'}(\Box \psi)\to u_{z}(\Box \psi)) \leq u_z(\Box\psi),
\]
which is an instance of a tautology in residuated lattices. 
\end{proof}
As a consequence of the truth lemma we get:
\begin{corollary}
	\begin{enumerate}
		\item The axiom $\Box\bot\vdash \bot$ is valid in $\mathbb{G}_{\mathbf{Fm}_{0\ast}}$.
		\item The axiom $\Box p \vdash p$ is valid in $\mathbb{G}_{\mathbf{Fm}_{1\ast}}$.
	\end{enumerate}
\end{corollary}
\begin{proof}
1. Clearly, since the axiom above does not contain atomic propositions, its validity on $\mathbb{G}_{\mathbf{Fm}_{0\ast}}$ coincides with its satisfaction on $\mathbb{M}_{\mathbf{Fm}_{0\ast}}$. Hence, $\mathbb{M}_{\mathbf{Fm}_{0\ast}}\models \Box\bot\vdash \bot$ iff $\descr{\bot}\subseteq \descr{\Box\bot}$ iff $\descr{\bot}(\alpha, z)\leq \descr{\Box\bot}(\alpha, z)$ for every $(\alpha, z)\in Z_X$, iff (by Lemma \ref{lemma:truthlemma}) $u_z(\bot)\to \alpha\leq u_z(\Box \bot)\to \alpha$, iff $u_z(\Box\bot)\leq u_z(\bot)$, which is true, since $u_z: \mathbf{Fm}_{0\ast}\to \mathbf{A}$ is order-preserving. 

2.  By Proposition \ref{prop:correspondence}.3 it is enough to show that $E \subseteq R_{Box}$ in $\mathbb{G}_{\mathbf{Fm}_{1\ast}}$, i.e.\ that $E(z,z') \leq R_{\Box}(z,z')$ for all $z,z' \in Z$.\footnote{Recall the setting of the present section is different from that in Appendix \ref{sec:correspondence} (cf.\ the discussion at the beginning of the present section). However, it is not difficult to verify that Proposition \ref{prop:correspondence}.3 holds verbatim also in the present setting, by suitably adapting the chains of equivalences used in the proof.} Recall that \[
E(z, z') = \bigwedge_{\phi\in \mathbf{Fm_{\ast}}}(f_z(\phi)\to u_{z'}(\phi))
\]
and 
\[
R_{\Box}(z, z') = \bigwedge_{\phi\in \mathbf{Fm_{\ast}}}(f_z(\phi)\to u^{-\Box}_{z'}(\phi)).
\]
Hence it is enough to show that $u_{z'}(\phi) \leq u^{-\Box}_{z'}(\phi) := \bigwedge \{u_{z'}(\psi) \mid \phi \vdash \Box \psi \}$, i.e. that if $\phi \vdash \Box \psi$, then $u_{z'}(\phi) \leq u_{z'}(\psi)$. Indeed, $\phi \vdash \Box \psi$ and  $\Box \psi \vdash \psi$ imply $\phi \vdash \psi$ and hence $u_{z'}(\phi) \leq u_{z'}(\psi)$ follows from the monotonicity of $u_{z'}$.
\end{proof}
	
\begin{theorem} \label{thm:completeness}
	For $i \in \{0,1\}$, the normal $\mathcal{L}$-logic $\mathbf{L}_i$ is sound and complete w.r.t.~its corresponding class of graph-based $\mathbf{A}$-frames. 
\end{theorem}
\begin{proof}
Consider an $\mathcal{L}$-sequent $\phi \vdash \psi$ that is not derivable in $\mathbf{L}_i$. In order to show that  $\mathbb{M}_{\mathbf{Fm}_{i}} \not\models \phi \vdash \psi$ (Definition \ref{def:Sequent:True:In:Model}), we  need to show that $\val{\phi}(z)\nleq\val{\psi}(z) = \descr{\psi}^{[0]}(z)$ for some  $z\in Z$.
Consider the proper filter $f_{\phi}$ and complement of proper ideal $u_{\psi}$ given by 
\[
f_{\phi}(\chi) = \left\{\begin{array}{ll}
1 &\text{if } \phi \vdash \chi\\
0 &\text{if } \phi \not \vdash \chi
\end{array} \right.
\]
and
\[
u_{\psi}(\chi) = \left\{\begin{array}{ll}
0 &\text{if } \chi \vdash \psi\\
1 &\text{if } \chi \not \vdash \psi
\end{array} \right.
\]
Then $\bigwedge_{\chi\in \mathbf{Fm}_i}(f_{\phi}(\chi)\to u_{\psi}(\chi)) = 1$, for else there would have to be a $\chi \in \mathbf{Fm}_i$ such that $f_{\phi}(\chi) = 1$ and $u_{\psi}(\chi) = 0$, which would mean that $\phi \vdash \chi$ and $\chi \vdash \psi$ and hence that $\phi \vdash \psi$, in contradiction with the assumption that $\phi \vdash \psi$ is not derivable. It follows that $z: = (f_{\phi}, u_{\psi})$ is a state in the canonical model $\mathbb{M}_{\mathbf{Fm}_i}$. 
By the Truth Lemma, $\val{\phi}(z) = f_{\phi}(\phi) = 1$, and moreover
\begin{center}
\begin{tabular}{cl}
   & $\descr{\psi}^{[0]}(z)$\\
= & $\bigwedge_{(\alpha, z')\in Z_X}\descr{\psi}(\alpha, z')\to (E(z, z')\to\alpha)$\\
$\leq$ & $\descr{\psi}(0, z)\to (E(z, z)\to 0)$\\
= & $(u_{\psi}(\psi) \to 0)\to (E(z, z)\to 0)$\\
= & $(0 \to 0)\to (1\to 0)$\\
 = & 0,
\end{tabular}
\end{center}
which proves the claim. 
\end{proof}
\end{document}